%% file: notes_char_var_v2_arxiv.tex
\documentclass[a4paper,english]{article}

\input{preamble_eng.tex}

\title{Character varieties for real forms of classical complex groups}
\author{
	\textsc{Miguel ACOSTA}
	\thanks{Miguel Acosta was partially supported by the grants R-STR-8023-00-B <<MnLU-MESR CAFE-AutoFi>> and R-AGR-3172-10-C <<FNR-OPEN>>.} 
	\footnote{
		Unité de recherche en Mathématiques,
		Université du Luxembourg, 
		Maison du Nombre,
		6, Avenue de la Fonte,
		L-4364 Esch-sur-Alzette,
		Luxembourg}
	 }

\begin{document}
\maketitle

\begin{abstract}
	Let $\Gamma$ be a finitely generated group, $G_\CC$ be a classical complex group and $G_\RR$ a real form of $G_\CC$. We propose a definition of the $G_\RR$-character variety of $\Gamma$ as a subset $\mathcal{X}_{G_\RR}(\Gamma)$ of the $G_\CC$-character variety $\mathcal{X}_{G_\CC}(\Gamma)$. We prove that these subsets cover the set of irreducible $G_\CC$-characters fixed by an anti-holomorphic involution $\Phi$ of $\mathcal{X}_{G_\CC}(\Gamma)$. Whenever $G_\RR$ is compact, we prove that $\mathcal{X}_{G_\RR}(\Gamma)$ is homeomorphic to the topological quotient $\mathrm{Hom}(\Gamma,G_\RR)/G_\RR$. Finally, we identify the reducible points of $\mathcal{X}_{\mathrm{GL}(n,\CC)}(\Gamma)$ fixed by an anti-holomorphic involution $\Phi$ as coming from direct sums of representations with values in real groups.
\end{abstract}

\section{Introduction}

 Let $\Gamma$ be a finitely generated group and $G$ be a Lie group.
 In many contexts, it is interesting to study the representations of $\Gamma$ with values in $G$ up to conjugation.
 A crucial example is when $\Gamma$ is the fundamental group of a manifold $M$ and we study the set of $(G,X)$-structures on $M$. In this setting, the space of structures is parametrized locally by the representations of $\Gamma$ with values in $G$ up to conjugation: it is the Ehresmann-Thusrton principle. See the articles of Canary, Epstein and Green \cite{canary_epstein_green} or Bergeron and Gelander \cite{bergeron_gelander} for a detailed exposition.
 However, in general, the space $\mathrm{Hom}(\Gamma,G)/G$, where $G$ acts by conjugation, may not be Hausdorff for the quotient topology, and has no particular structure, so it is difficult to study.

 When $G$ is a complex algebraic reductive group, the set $\mathrm{Hom}(\Gamma,G)$ is an algebraic set, and the geometric invariant theory (GIT) allows us to consider an algebraic quotient, that is called the character variety:
 \[ \mathcal{X}_G(\Gamma) = \mathrm{Hom}(\Gamma,G)//G .\]
 For the details of the construction, see for example the article of Lubotzki and Magid \cite{lubotzky_varieties_1985}, the survey of Sikora \cite{sikora_character_2012a} or the background section of \cite{heusener_slnc_2016}. This quotient can be considered as a possibly non-reduced scheme $\mathfrak{X}_G(\Gamma)$ or as an algebraic set $\mathcal{X}_G(\Gamma)$, that coincides with the Hausdorffization of the topological quotient $\mathrm{Hom}(\Gamma,G)/G$. We will mainly keep this last point of view in this article.
 
 When $G$ does not satisfy these last hypotheses, it is then natural to try to construct a similar object that would allow to study representations up to conjugation, and to relate it to a character variety. In this article, we will consider the frame of real forms $G_\RR$ of classical complex groups $G_\CC$. In this case, the character variety $\mathcal{X}_{G_\CC}(\Gamma)$ is well defined as a GIT quotient.

 Another case of interest is when $G$ is a compact group, so the topological quotient $\mathrm{Hom}(\Gamma,G)/G$ is also compact. When $G= \mathrm{SU}(n)$,
 Procesi and Schwartz prove in \cite{procesi_inequalities_1985} that $\mathrm{Hom}(\Gamma,G)/G$ is a semi-algebraic set; we prove in \cite{acosta_character_2019} that it can be embedded in the $\mathrm{SL}(n,\CC)$-character variety of $\Gamma$. 
 Furthermore, if $K$ is a maximal compact subgroup of a complex or real algebraic reductive group $G$, some
 recent results prove there is a strong deformation retraction from the set of closed orbits of $\mathrm{Hom}(\Gamma , G)/G$ to the quotient $\mathrm{Hom}(\Gamma , K )/K$ . This fact is proven for Abelian groups by Florentino and Lawton in \cite{florentino_topology_2014}, for free groups by Casimiro, Florentino, Lawton and Oliveira in \cite{casimiro_topology_2014}, and for nilpotent groups by Bergeron in \cite{bergeron_topology_2015}.

 In \cite{acosta_character_2019}, we propose a definition of \emph{character variety for a real form $G_\RR$ of $\mathrm{SL}(n,\CC)$} as the image of $\mathrm{Hom}(\Gamma , G_\RR)$ by the natural projection $\mathrm{Hom}(\Gamma , G_\RR) \to \mathcal{X}_{\mathrm{SL}(n,\CC)}(\Gamma)$. We consider two anti-holomorphic involutions, $\Phi_1$ and $\Phi_2$, of $\mathcal{X}_{\mathrm{SL}(n,\CC)}(\Gamma)$ and prove that irreducible characters fixed by $\Phi_2$ lift to representations with values in a $\mathrm{SU}(p,q)$ group, and that irreducible characters fixed by $\Phi_1$ lift to a representation with values in a $\mathrm{SL}(n,\RR)$ or $\mathrm{SL}(n/2 , \HH)$ if $n$ is even.
Moreover, we prove that if the real form is $\mathrm{SU}(n)$, the $\mathrm{SU}(n)$-character variety, as a subset of the $\mathrm{SL}(n,\CC)$-character variety, is homeomorphic to the topological quotient $\mathrm{Hom}(\Gamma,\mathrm{SU}(n))/\mathrm{SU}(n)$.
In this article, we generalize these results in three different directions. First, we extend the definition and the lifting of irreducible characters to the classical complex groups, namely $\mathrm{GL}(n,\CC)$, $\mathrm{SL}(n,\CC)$, $\mathrm{O}(n,\CC)$, $\mathrm{SO}(n,\CC)$ and  $\mathrm{Sp}(2n,\CC)$. We prove the following theorem:

\begin{thm}\label{thm:main_thm}
	Let $n$ be a positive integer, $\Gamma$ a finitely generated group and $G_{\CC}$ be one of the groups $\mathrm{GL}(n,\CC)$, $\mathrm{SL}(n,\CC)$, $\mathrm{O}(n,\CC)$, $\mathrm{SO}(n,\CC)$ or $\mathrm{Sp}(2n,\CC)$. Let $\Phi$ be the anti-holomorphic involution of $\mathcal{X}_{G_{\CC}}(\Gamma)$ induced by one of the involutions $A \mapsto \conjug{A}$ or $A \mapsto \transp{\conjug{A}}^{-1}$. Let $\rho \in \mathrm{Hom}(\Gamma,G_{\CC})$ be an irreducible representation such that $\Phi(\chi_\rho) = \chi_\rho$.
	Then, there exist a real form $G_{\RR}$ of $G_{\CC}$, $\sigma \in \mathrm{Hom}(\Gamma,G_{\RR})$ and $P \in G_{\CC}$ such that
	\begin{align*}
	\forall \gamma \in \Gamma &: P\rho(\gamma)P^{-1} = \sigma(\gamma).
	\end{align*}
\end{thm}
In the statement, an irreducible representation should be understood as a representation $\rho: \Gamma \to G_\CC \subset \mathrm{GL}(n,\CC)$ such that $\Im(\rho)$ has no non-trivial stable subspace for its natural action on $\CC^n$. In the same way, a real form of $G_\CC$ should be understood as one of the groups listed in \Cref{table:real_forms}.

Then, we prove the same homeomorphism result for the character varieties for compact forms of classical complex groups; which is summarized in the following proposition:
\begin{prop*}[\Cref{prop:character_compact}]
	Let $\Gamma$ be a finitely generated group and let $\mathrm{G}_\CC$ be one of the groups $\mathrm{GL}(n,\CC)$, $\mathrm{SL}(n,\CC), \mathrm{O}(n,\CC), \mathrm{SO}(n,\CC)$ and $\mathrm{Sp}(2n,\CC)$, and $K$ its compact real form.
	Then $\mathcal{X}_K(\Gamma)$ is homeomorphic to the topological quotient $\mathrm{Hom}(\Gamma,K)/K$.
\end{prop*}

Finally, we consider characters of reducible representations in the $\mathrm{GL}(n,\CC)$-character variety. On the one hand, we prove that if a character is fixed by $\Phi_1$, it lifts to a direct sum of representations with values in some $\mathrm{GL}(r_1 , \RR)$ and some $\mathrm{GL}(r_2 , \HH)$. On the other hand, if a character is fixed by $\Phi_2$, it lifts to a semi-simple representation with values in some $\mathrm{U}(p,q)$. This last result can be stated as follows:
\begin{thm*}[\Cref{thm:char_var_unitary}]
	Let $\Gamma$ be a finitely generated group and $n$ be a positive integer. Then
	\begin{equation*}
	\mathcal{X}^{\Phi_2}_{\mathrm{GL}(n,\CC)}(\Gamma) = 
	\bigcup_{p+q = n} \mathcal{X}_{\mathrm{U}(p,q)}(\Gamma).
	\end{equation*}
\end{thm*}

\paragraph*{Outline of the article:}
In Section 2, we fix notation for the classical complex groups and their real forms.
In Section 3, we recall some background material on the character varieties for the classical complex groups, and define two anti-holomorphic involutions on them. Moreover, if $G_\RR$ is a real form of a classical complex group $G_\CC$ we propose a definition for the $G_\RR$-character variety of $\Gamma$ as a subset of the $G_\CC$-character variety.
We devote Section 4 to the proof of \Cref{thm:main_thm}, giving first a short proof of a similar statement, and then dealing separately with symplectic and orthogonal groups for an elementary proof of the theorem.
In Section 5, we consider the compact real forms, and prove in \Cref{prop:character_compact} that the character variety for the compact form $K$ of a classical complex group is homeomorphic to the topological quotient $\mathrm{Hom}(\Gamma,K)/K$. Finally, in Section 6, we consider reducible representations and prove \Cref{thm:char_var_unitary}.

\paragraph*{Acknowledgments:} The author would like to thank Clément Guérin and Alain Genestier for many conversations.

\section{Notation and real forms of classical groups}
Here, the classical complex Lie groups are the groups $\mathrm{GL}(n,\CC)$, $\mathrm{SL}(n,\CC)$, $\mathrm{O}(n,\CC)$, $\mathrm{SO}(n,\CC)$ and  $\mathrm{Sp}(2n,\CC)$, where $n$ is a positive integer. We fix notation for these groups, and recall their real forms, which are classical real Lie groups.

\subsection{Notation and classical complex groups}

For $n \in \NN$, we let $I_n$ denote the identity matrix of size $n \times n$, and, for $p,q,n \in \NN$, we define the matrices $J_{2n}$, $I_{p,q}$ and $K_{p,q}$ by blocks as:
\begin{align*}
	J_{2n} &=
	\begin{pmatrix}
		0 & I_n \\
		-I_n & 0
	\end{pmatrix}
	&
	I_{p,q} &=
	\begin{pmatrix}
		I_p & 0 \\
		0 & -I_q
	\end{pmatrix}
	&
	K_{p,q} &=
	\begin{pmatrix}
		I_{p,q} & 0 \\
		0 & I_{p,q}
	\end{pmatrix}
\end{align*}
If there is no ambiguity for the integer $n$, we will often write $I$ instead of $I_n$ and $J$ instead of $J_{2n}$. 

Recall that $\mathrm{GL}(n,\CC)$ is the group of $n\times n$ invertible complex matrices and $\mathrm{SL}(n,\CC)$ is the subgroup of $\mathrm{GL}(n,\CC)$ of the matrices with determinant $1$.
The orthogonal group $\mathrm{O}(n,\CC)$ is the subgroup of matrices $M \in \mathrm{GL}(2n,\CC)$ satisfying $\transp{M}M = I$. It is the group of matrices of $\mathrm{GL}(2n , \CC)$ preserving the symmetric bilinear form defined by $I$. The special orthogonal group $\mathrm{SO}(n,\CC)$ is defined as $\mathrm{SL}(n,\CC) \cap \mathrm{O}(n,\CC)$.
At last, the symplectic group $\mathrm{Sp}(2n,\CC)$ is the subgroup of matrices $M \in \mathrm{SL}(2n , \CC)$ satisfying $\transp{M}JM = J$. It is the group of matrices of $\mathrm{GL}(2n , \CC)$ preserving the skew-symmetric form defined by $J$.

Observe that when we will have to deal with a bilinear or Hermitian form, we will abusively identify it to its matrix in the canonical basis of $\CC^n$. If $M \in \mathrm{GL}(n,\CC)$, we let $M^*$ or $\transp{\conjug{M}}$ denote the adjoint of $M$ for the usual Hermitian product.

\subsection{Real forms of classical complex groups}
 For a classical group $G_\CC$, there are several ways to define and classify its real forms, which agree most of the time.
 It is possible to define a real form of $G_\CC$ as the fixed points of an anti-holomorphic involution. In this case, we can consider real forms up to an automorphism of $G_\CC$. However, since we will consider objects up to conjugation in $G_\CC$ for the character varieties, we shall consider real forms up to \emph{inner} automorphisms. As pointed out by Adams in \cite{adams_guide_2008}, these two definitions give the same objects most of the time. In our case, the only difference will come from the group $\mathrm{SO}(4m,\HH)$, that we define later. 
 
 Another possibility is to consider the real forms of classical groups as defined by Onishchik and Vinberg in \cite[p. 101]{onishchik_lie_1994}.
 If $G_\CC$ is a complex algebraic subgroup, a real algebraic subgroup $G_\RR$ is said to be a \emph{real form of $G_\CC$} if the identity embedding $G_\RR \subset G_\CC$ extends to an isomorphism $G_\RR \otimes \CC \to G_\CC$. Again, we can consider such real groups up to automorphisms or up to inner automorphisms. The two definitions agree, and agree with the fixed points of an anti-holomorphic involution, for all our examples besides the groups $\mathrm{SO}(2m,\HH)$; we will state a remark about this point.
 
 In any case, there is an explicit list of representatives of real forms of $G_\CC$. For our statements, a real form of a classical group will be an element of \Cref{table:real_forms}. In the rest of the section, we recall the definitions of the groups of the table. For details on the classification at the level of Lie algebras, see for example the books of Knapp \cite{knapp_lie_2002} , Helgason \cite{helgason_geometric_2008}, or Onishchik and Vinberg \cite{onishchik_lie_1994}. 

 \begin{table}[htb]
 	\centering
 	\begin{tabular}{|r|ll|}
 		\hline
 		Complex group $G_\CC$ & Real forms $G_\RR \subset G_\CC$ & \\
 		\hline
 		$\mathrm{GL}(2m+1,\CC)$ & $\mathrm{GL}(2m+1,\RR) , \mathrm{U}(p,q)$ &(for $p\geq q$ and $p+q = 2m+1$)\\
 		$\mathrm{GL}(2m,\CC)$ & $\mathrm{GL}(2m,\RR) , \mathrm{GL}(m,\HH), \mathrm{U}(p,q)$ &(for $p\geq q$ and $p+q = 2m$)\\
 		$\mathrm{SL}(2m+1,\CC)$ & $\mathrm{SL}(2m+1,\RR) , \mathrm{SU}(p,q)$ &(for $p\geq q$ and $p+q = 2m+1$)\\
 		$\mathrm{SL}(2m,\CC)$ & $\mathrm{SL}(2m,\RR) , \mathrm{SL}(m,\HH), \mathrm{SU}(p,q)$ &(for $p\geq q$ and $p+q = 2m$)\\
 		$\mathrm{Sp}(2m,\CC)$ & $\mathrm{Sp}(2m,\RR) , \mathrm{Sp}(2p,2q)$ &(for $p\geq q$ and $p+q = m$)\\
 		$\mathrm{O}(2m+1,\CC)$ & $ \mathcal{O}(p,q)$& (for $p\geq q$ and $p+q = 2m+1$)\\
 		$\mathrm{O}(2m,\CC)$ & $ \mathcal{O}(p,q) , \mathrm{O}(m,\HH)$& (for $p\geq q$ and $p+q = 2m$)\\
 		$\mathrm{SO}(2m+1,\CC)$ & $\mathcal{SO}(p,q)$& (for $p\geq q$ and $p+q = 2m+1$)\\
 		$\mathrm{SO}(4m+2,\CC)$ & $ \mathcal{SO}(p,q) , \mathrm{SO}(2m+1,\HH)$& (for $p\geq q$ and $p+q = 4m+2$)\\
 		$\mathrm{SO}(4m,\CC)$ & $ \mathcal{SO}(p,q) , \mathrm{SO}(2m,\HH) , \mathrm{SO}^{-}(2m,\HH)$& (for $p\geq q$ and $p+q = 4m$)\\
 		\hline
 	\end{tabular}
 	\caption{Real forms of classical Lie groups} \label{table:real_forms}
 \end{table}

\subsubsection{Linear groups and special linear groups}
If $n$ is odd, the real forms of $\mathrm{GL}(n,\CC)$ up to inner automorphism are $\mathrm{GL}(n,\RR)$ and the unitary groups $\mathrm{U}(p,q)$, where $p \geq q$ are non-negative integers such that $p+q = n$. These last groups are the groups preserving a Hermitian form with signature $(p,q)$. If $n=2m$ is even, the same groups are real forms of $\mathrm{GL}(n,\CC)$, but there is another one: the quaternionic group $\mathrm{GL}(m,\HH)$.
The real forms of $\mathrm{SL}(n,\CC)$ are the intersections of the real forms of $\mathrm{GL}(n,\CC)$ with $\mathrm{SL}(n,\CC)$. We let them be denoted $\mathrm{SL}(n,\RR)$, $\mathrm{SU}(p,q)$ and, when $n=2m$ is even, $\mathrm{SL}(m,\HH)$.
The groups are defined as follows:
\begin{align*}
	\mathrm{U}(p,q) &= \{M \in \mathrm{GL}(n,\CC) \mid \transp{\conjug{M}}I_{p,q}M = I_{p,q} \} 
	&
	\mathrm{SU}(p,q) &= \mathrm{U}(p,q) \cap \mathrm{SL}(n,\CC)
	\\
	\mathrm{GL}(m,\HH) &= \{ M \in \mathrm{GL}(2m , \CC) \mid MJ = J\conjug{M}  \}
	&
	\mathrm{SL}(m,\HH) &= \mathrm{GL}(m,\HH) \cap \mathrm{SL}(2m , \CC)
\end{align*}

The compact forms of $\mathrm{GL}(n,\CC)$ and $\mathrm{SL}(n,\CC)$ are respectively $\mathrm{U}(n,0)$, and $\mathrm{SU}(n,0)$, found also as $\mathrm{U}(n)$ and $\mathrm{SU}(n)$, while the split forms are $\mathrm{GL}(n,\RR)$ and $\mathrm{SL}(n,\RR)$.

\begin{rem}
	The matrices of $\mathrm{GL}(m,\HH)$ are the invertible matrices of the form $
	\begin{pmatrix}
	A & -\conjug{B} \\
	B & \conjug{A}
	\end{pmatrix}
	$
	for $A,B \in \mathcal{M}_m(\CC)$. Hence, if $M \in \mathrm{GL}(m,\HH)$, then $\det(M)$ is a positive real number. See for example \cite[Proposition 4.2]{ZhangQuaternionsmatricesquaternions1997} for a proof.
\end{rem}

\subsubsection{Symplectic groups}
For $n \in \NN$ and $p,q \in \NN$ such that $p\geq q$ and $p+q = n$ we define the groups $\mathrm{Sp}(2n,\RR)$ and $\mathrm{Sp}(2p,2q)$ as follows:
\begin{align*}
	\mathrm{Sp}(2n,\RR) &= \left\{M \in \mathrm{Sp}(2n,\CC)  \mid M \in \mathcal{M}_n(\RR)  \right\} \\
	\mathrm{Sp}(2p,2q) &= \left\{M \in \mathrm{Sp}(2n,\CC)  \mid \transp{\conjug{M}}K_{p,q}M = K_{p,q} \right\}
\end{align*} 

The real forms of the complex Lie group $\mathrm{Sp}(2n,\CC)$ up to inner automorphism are precisely the groups $\mathrm{Sp}(2n,\RR)$ and $\mathrm{Sp}(2p,2q)$. The compact form of the group is $\mathrm{Sp}(2n,0)$, denoted sometimes $\mathrm{Sp}(n)$ or $\mathrm{USp}(n)$, while the split form is $\mathrm{Sp}(2n,\RR)$. 
We conclude this subsection by two remarks that might be useful later.

\begin{rem}
	We have
	$\mathrm{Sp}(2n,0) = \mathrm{Sp}(2n,\CC) \cap \mathrm{U}(2n) = \mathrm{Sp}(2n,\CC) \cap \mathrm{SL}(n,\HH) = \mathrm{SL}(n,\HH) \cap \mathrm{U}(2n)$.
\end{rem}

\begin{rem}
	If $G$ is a real form of $\mathrm{Sp}(2n,\CC)$, then $G$ preserves a Hermitian form.
\end{rem}
\begin{proof}
	The remark is immediate for the groups $\mathrm{Sp}(2p,2q)$, which are defined as preserving the Hermitian form of matrix $K_{p,q}$. Since a matrix $M\in \mathrm{Sp}(2n,\RR)$ is real, it satisfies $\transp{\conjug{M}}JM = J$, and therefore $\transp{\conjug{M}}(iJ)M = iJ$. since the matrix $iJ$ is Hermitian, the group $\mathrm{Sp}(2n,\RR)$ preserves the Hermitian form $iJ$, which has signature $(n,n)$.
\end{proof}

\subsubsection{Orthogonal groups}

The real forms of $\mathrm{O}(n, \CC)$ up to isomorphism  are the groups $\mathrm{O}(p,q)$ for $p+q = n$ when $n$ is odd. These are the real groups preserving a real quadratic form of signature $(p,q)$. If $n=2m$ is even, there is another real form, namely the quaternionic group $\mathrm{O}(m,\HH)$. The real forms of the group $\mathrm{SO}(n,\CC)$, up to isomorphism, are the intersections with $\mathrm{SL}(n,\CC)$ of the real forms of $\mathrm{O}(n,\CC)$.
These groups are defined as follows:
\begin{align*}
	\mathrm{O}(p,q) &= \{M\in\mathrm{GL}(n,\RR) \mid \transp{M}I_{p,q}M = I_{p,q} \}
	&
	\mathrm{SO}(p,q) &= \mathrm{O}(p,q) \cap \mathrm{SL}(n,\CC)
	\\
	\mathrm{O}(m,\HH) &= \{M\in\mathrm{O}(n,\CC) \mid JM = \conjug{M}J \}  
	&
	\mathrm{SO}(m,\HH) &= \mathrm{O}(m,\HH) \cap \mathrm{SL}(2m , \CC)
\end{align*}

We let $\mathrm{O}(n)$ and $\mathrm{SO}(n)$ denote the groups $\mathrm{O}(n,0)$ and $\mathrm{SO}(n,0)$; they are the compact forms of $\mathrm{O}(n,\CC)$ and $\mathrm{SO}(n,\CC)$. The split forms are $\mathrm{O}(m+1,m)$ and $\mathrm{SO}(m,m+1)$ if $n=2m+1$ is odd and $\mathrm{O}(m,m)$ and $\mathrm{SO}(m,m)$ if $n=2m$ is even. Notice also that the notation for the group $\mathrm{SO}(m,\HH)$ is not standard, other common notations for the group are $\mathrm{SO}^*(2n)$ are $\mathrm{SU}^*(2n)$.

 \begin{rem}
 	Since the matrices in $\mathrm{GL}(n,\mathbb{H})$ have positive determinant, with our definition we have $\mathrm{O}(n,\HH) = \mathrm{SO}(n,\HH)$. We will keep the most common notation $\mathrm{SO}(n,\HH)$ for the group, but we will consider it at the same time as a real form of $\mathrm{O}(2n,\CC)$ and of $\mathrm{SO}(2n,\CC)$. It is, indeed, the set of fixed points of an anti-holomorphic involution of $\mathrm{O}(2n,\CC)$, but $\mathrm{SO}(n,\HH)\otimes \CC \simeq \mathrm{SO}(2n,\CC)$.
 \end{rem}

Since the definitions above do not give subgroups of $\mathrm{O}(n,\CC)$ and $\mathrm{SO}(n,\CC)$ for some real forms, we need to set some more notation in order to work with actual subgroups. We will keep the usual definition of the groups $\mathrm{O}(p,q)$ and $\mathrm{SO}(p,q)$. However, we will consider the copies of these groups obtained by conjugation by a matrix $D_{p,q}$ in order to deal with subgroups of $\mathrm{O}(n,\CC)$.  Let $D_{p,q}$ be the matrix
\begin{align*}
	D_{p,q}
	&=
	\begin{pmatrix}
		I_p & 0 \\
		0 & iI_q
	\end{pmatrix}.
\end{align*}
Observe that $\transp{D_{p,q}}I_{p,q}D_{p,q} = I_{2n+1}$.
Thus, we define the groups $\mathcal{O}_{p,q}$ and $\mathcal{SO}_{p,q}$ as follows:
\begin{align*}
	\mathcal{O}_{p,q} &= \left\{ D_{p,q}MD_{p,q}^{-1} \mid M \in \mathrm{O}(p,q) \right\} \subset \mathrm{O}(n,\CC) \\
	\mathcal{SO}_{p,q} &= \left\{ D_{p,q}MD_{p,q}^{-1} \mid M \in \mathrm{SO}(p,q) \right\} \subset \mathrm{SO}(n,\CC).
\end{align*}

 Finally, we will need to define an extra real group in dimension $n = 4m$. 
 	Let $P_0 \in \mathrm{O}(4m,\CC) \setminus  \mathrm{SO}(4m,\CC)$ be a matrix fixed once and for all. We define
 	$\mathrm{SO}^-(2m,\HH)$
 	as follows
 	\[
 	\mathrm{SO}^-(2m,\HH) = \{ P_0 M P_0^{-1} \mid M \in \mathrm{SO}(2m,\HH)   \}
 	\]
 \begin{rem}
  By definition, the groups $\mathrm{SO}(2m,\HH)$ and $\mathrm{SO}^{-}(2m,\HH)$ are conjugated in $\mathrm{O}(4m,\CC)$. However, they are not conjugated in $\mathrm{SO}(4m,\CC)$.
 \end{rem}	
\begin{proof}
	If they were, there would be a matrix $K \in \mathrm{O}(4m,\CC)$ with determinant $-1$ that stabilizes $\mathrm{SO}(2m,\HH)$ by conjugation. Then, for all $M \in \mathrm{SO}(2m,\HH)$, we have $\conjug{(KMK^{-1})}J = J (KMK^{-1})$, so
	$\conjug{M}(\conjug{K}^{-1}JK) = (\conjug{K}^{-1} J K) M$.
	Hence, $(\conjug{K}^{-1} J K)J^{-1}$ commutes with $\mathrm{SO}(2m,\HH)$, and therefore is a homothety. Thus, there is $\lambda \in \CC$ such that $\conjug{K}^{-1} J K = \lambda J$. Since $\conjug{K}^{-1} J K \in \mathrm{O}(4m,\CC)$, we know that $\lambda = \pm 1$.
	
	If $\lambda = 1$, then $K \in \mathrm{SO}(2m,\HH)$ and has positive determinant, which is a contradiction.
	Now suppose that $\lambda=-1$, so $\conjug{K} J = -J K$. Consider the matrix by blocks
	$K' =
	\begin{pmatrix}
	0 & I_{2m} \\
	I_{2m} & 0
	\end{pmatrix}
	$. It is real, has determinant $1$ and satisfies $K'J = -JK'$. Thus, $(KK')J = J\conjug{(KK')}$. Hence, $KK' \in \mathrm{GL}(2m,\HH)$ and has positive determinant. Since $\det(K')=1$, we have $\det(K)>0$, which is also a contradiction.
\end{proof} 

 Observe that in dimension $4m+2$, there is no need for considering two copies of $\mathrm{SO}(2m+1,\HH)$, since they are all conjugated in $\mathrm{SO}(4m+2,\CC)$.
 Indeed, the matrix $
 \begin{pmatrix}
 0 & I_{2m+1} \\
 I_{2m+1} & 0
 \end{pmatrix}
 $
 is orthogonal, has determinant $-1$ and stabilizes $\mathrm{SO}(2m+1,\HH)$ by conjugation.

\section{Character varieties for classical groups}

In this section, we describe the construction and state some facts about the character varieties for classical groups, namely $\mathrm{GL}(n,\CC)$, $\mathrm{SL}(n,\CC)$, $\mathrm{Sp}(2n,\CC)$, $\mathrm{O}(n,\CC)$ and $\mathrm{SO}(n,\CC)$. For the latter two, we will need to distinguish the cases where $n$ is even or odd. Observe that the classical groups are all algebraic subgroups of some $\mathrm{GL}(n,\CC)$.

Let $\Gamma$ be a finitely generated group, and $G_\CC$ be a classical group. If $\Gamma$ is generated by $\{\gamma_1, \dots , \gamma_s \}$, we can consider
$\mathrm{Hom}(\Gamma,G_\CC)$ as an algebraic set
by identifying it to the image of the map
\[
\begin{aligned}
\mathrm{Hom}(\Gamma,G_\CC) &\to (G_\CC)^s\\
\rho & \mapsto (\rho(\gamma_1), \dots \rho(\gamma_s))
\end{aligned}
\]
Indeed, the map is injective since $\{\gamma_1, \dots , \gamma_s \}$ generates $\Gamma$; and the image is an algebraic set, since the relations of the group become polynomial equations. A change on the set of generators leads to a polynomial map between the two images, so the algebraic structure of $\mathrm{Hom}(\Gamma,G_\CC)$ is well defined and independent from the choice of generators.

We would like to consider representations up to conjugation in $G_\CC$. The group $G_\CC$ acts on $\mathrm{Hom}(\Gamma,G_\CC)$ by conjugation, but the usual quotient $\mathrm{Hom}(\Gamma,G_\CC)/G_\CC$ does not behave well in general; in general it is not a Hausdorff space.
The $G_\CC$-character variety of $\Gamma$ is defined as the quotient of $\mathrm{Hom}(\Gamma  ,G_\CC)$ by the conjugation action of $G_\CC$, but in the category of algebraic sets instead of the category of topological spaces.
The Geometric Invariant Theory (GIT) theory ensures that this quotient is well defined and its ring of functions is the ring of functions of $\mathrm{Hom}(\Gamma  ,G_\CC)$ which are invariant by conjugation by elements of $G_\CC$.
We let $\mathrm{Hom}(\Gamma,G_\CC)//G_\CC$ or $\mathcal{X}_{G_\CC}(\Gamma)$ denote it.
For a complete construction, see for example the review of Sikora
\cite{sikora_character_2012a}.

However, it is sometimes useful to consider the character variety as a possible non reduced affine scheme, as described also in
\cite{sikora_character_2012a}
 or
\cite{marche_sl2character_2016}.
Let $\CC [\mathrm{Hom}(\Gamma ,G_\CC)]$ be the ring of functions of $\mathrm{Hom}(\Gamma ,G_\CC)$ and $\CC [\mathrm{Hom}(\Gamma ,G_\CC)]^{G_\CC}$ the ring of invariant functions by the action by conjugation of $G_\CC$. Then, let $\mathfrak{X}_{G_\CC} (\Gamma)$ be the affine scheme $\mathrm{Spec}(\CC [\mathrm{Hom}(\Gamma ,G_\CC)]^{G_\CC})$.

Following Sikora in \cite{sikora_character_2012a} or the background given by Heusener in
\cite{heusener_slnc_2016}, we will mainly consider character varieties as algebraic sets.
The character variety $\mathcal{X}_{G_\CC}(\Gamma)$, as an algebraic set, is the set of complex points of $\mathfrak{X}_{G_\CC}(\Gamma)$ and corresponds to the set of closed orbits of $\mathrm{Hom}(\Gamma,G_\CC)$ for the action of $G_\CC$ by conjugation. These orbits are precisely the ones of poly-stable representations. In the terms of Sikora in \cite{sikora_character_2012a}, the poly-stable representations are precisely the \emph{completely reducible} ones. Thus, above each point of the character variety $\mathcal{X}_{G_\CC}(\Gamma)$ there is a unique completely reducible representation up to conjugacy.

 When the group $G_\CC$ is $\mathrm{GL}(n,\CC)$, this notion coincides with the semi-simplicity of a representation; we use the same terminology as Heusener in \cite{heusener_slnc_2016}. We say that $\rho \in \mathrm{Hom}(\Gamma, \mathrm{GL}(n,\CC) )$ is \emph{irreducible} if the only stable subspaces of $\CC^n$ by $\Im(\rho)$ are $\CC^n$ and $\{0\}$. We say that a representation $\rho \in \mathrm{Hom}(\Gamma, \mathrm{GL}(n,\CC) )$ is \emph{semi-simple} if it is a direct sum of irreducible representations. Since the classical complex groups $G_\CC$ are defined as subgroups of some $\mathrm{GL}(n,\CC)$, we will say that a representation $\rho \in \mathrm{Hom}(\Gamma,G_\CC)$ is \emph{irreducible} if it is irreducible as a representation in $\mathrm{Hom}(\Gamma, \mathrm{GL}(n,\CC) )$. This definition implies other usual notions of irreducible, such as not being contained in a proper parabolic subgroup.

Besides this section, we will always consider the character varieties as affine algebraic sets, and not as schemes. However, a generating set for the algebra of invariant functions gives coordinates for the character  variety $\mathcal{X}_{G_\CC}(\Gamma)$. Indeed, if $\{f_1, \dots , f_k\}$ is a generating set, the character variety $\mathcal{X}_{G_\CC}(\Gamma)$ is identified with the image of the map
\[
\begin{aligned}
\mathrm{Hom}(\Gamma,G_\CC) &\to \CC^k\\
\rho & \mapsto (f_1(\rho), \dots , f_k(\rho))
\end{aligned}
\]
Observe that the map factors through $\mathcal{X}_{G_\CC}(\Gamma)$ since the functions $f_i$ are invariant by conjugation. The map gives therefore an injective polynomial map from $\mathcal{X}_{G_\CC}(\Gamma)$ to $\CC^k$. For $\rho \in \mathrm{Hom}(\Gamma,G_\CC)$, we let $\chi_\rho$ denote the image of rho in $\mathcal{X}_{G_\CC}(\Gamma)$. We will often identify it, abusively, with the corresponding point of $\CC^k$ when choosing coordinates for the character variety.

Before describing the coordinates for character varieties for each classical group, we need to make a last observation. If $\Gamma$ is generated by $s$ elements and $F_s$ is the free group of rank $s$, then $\mathrm{Hom}(\Gamma,G_\CC) \subset \mathrm{Hom}( F_s, G_\CC )$ is Zariski-closed. Since the conjugation by $G_\CC$ is the same for both representation varieties, we have an inclusion $\mathcal{X}_{G_\CC}(\Gamma) \subset \mathcal{X}_{G_\CC}(F_s)$. Thus, a set of coordinate functions for $F_s$-representations will also be a set of coordinate functions for $\Gamma$-representations.

In the rest of the section, we will describe generating sets of functions for $\mathfrak{X}_{G_\CC}(F_s)$, and thus coordinates for the character varieties $\mathcal{X}_{G_\CC}(\Gamma)$.

\subsection{$G_\CC$-character varieties}
 We consider here the character varieties for almost all the classical groups. Namely, in this subsection $G_\CC$ will be one of the groups $\mathrm{GL}(n,\CC)$, $\mathrm{SL}(n,\CC)$, $\mathrm{O}(n,\CC)$, $\mathrm{Sp}(2n,\CC)$ and $\mathrm{SO}(2m+1,\CC)$. The remaining cases are the even orthogonal groups $\mathrm{SO}(2m,\CC)$, that we discuss in the next section.
 
 In the case considered here, the trace functions play a capital role, since the algebra of invariant functions is generated by them. 
 Consider the trace functions defined as follows:
 
\begin{defn}
	For $\gamma \in \Gamma$, we call the \emph{trace function of $\gamma$} the function
	\[\tau_\gamma: 
	\begin{aligned}
	\mathrm{Hom}(\Gamma,G_\CC) &\to \CC\\
	\rho &\mapsto \mathrm{tr}(\rho(\gamma))
	\end{aligned}\]
	It is an invariant function for $\mathrm{GL}(n,\CC)$-conjugation.
\end{defn}

 We have also the following theorem, about the algebra of invariant functions in the case of representations of free groups. 

\begin{thm}
	Let $m,n \in \NN$. If $F_m$ is the free group of order $m$, then the ring $\CC  [\mathrm{Hom}(F_m , G_\CC ) ]^{G_\CC}$ is generated by a finite number of trace functions.
\end{thm}

The result is proved by Procesi in \cite{procesi_invariant_1976} for the groups $\mathrm{SL}(n,\CC)$ and $\mathrm{GL}(n,\CC)$ (Theorem 1.3), $O_n(\CC)$ (Theorem 7.1) and $\mathrm{Sp}(2n,\CC)$ (Theorem 10.1).
In \cite{sikora_generating_2013}, Sikora shows that for a finitely generated group $\Gamma$ and $n\in \NN$, the ring of functions of the $\mathrm{SO}(2n+1,\CC)$-character varieties of $\Gamma$ are generated by the trace functions, as a consequence of the results of Procesi.
More precisely, in the case of $\mathrm{GL}(n,\CC)$ and $\mathrm{SL}(n,\CC)$ Procesi gives in \cite[Theorem 3.4]{procesi_invariant_1976} an explicit set of generators by considering the words of length $\leq 2^n-1$ in the generators of $F_m$. In the same article, he also gives generating sets for orthogonal and symplectic character varieties with the same type of bounds on length. In \cite{sikora_generating_2013}, Sikora gives more efficient generating sets.

Therefore, if $\gamma_1,\dots,\gamma_s \in \Gamma$ are such that $\tau_{\gamma_1}, \dots , \tau_{\gamma_s}$ generate $\CC [\mathrm{Hom}(F_m ,\mathrm{GL}(n,\CC)]^{G_\CC}$, then we can identify $\mathcal{X}_{G_\CC}(\Gamma)$ with the image of the map
\[
\begin{array}{rcl}
\mathrm{Hom}(\Gamma, G_\CC) &\to& \CC^s \\
\rho &\mapsto& (\mathrm{tr}(\rho(\gamma_1)), \dots , \mathrm{tr}(\rho(\gamma_s))
\end{array}
\]

Since we have an inclusion of character varieties of finitely generated groups into character varieties of free groups, the result also holds for finitely generated groups.

\subsection{SO(2n,C)-character varieties}
For $\mathrm{SO}(2n,\CC)$-representations, the set of trace functions does no longer generate the ring of invariant functions, as pointed out by Sikora in \cite{sikora_so2nccharacter_2017}. However, following \cite{sikora_generating_2013}, generating sets for the algebra of invariant functions are still known in this case.
Consider the function $Q_{2n} : (\mathcal{M}_{2n}(\CC))^n \to \CC$, defined by
\[
Q_{2n}(A_1,\dots,A_m) = 
\sum_{\sigma \in \mathfrak{S}_{2m}} 
\epsilon(\sigma)
\prod_{i=1}^{m}
((A_i)_{\sigma(2i-1)\sigma(2i)}
-(A_i)_{\sigma(2i)\sigma(2i-1)}).
\]

If $P, A_1, \dots , A_m \in \mathrm{O}(2m,\CC)$, then $Q_m(PA_1P^{-1},\dots,PA_mP^{-1}) =\det(P)Q_m(A_1,\dots,A_m)$. In particular, the function $Q_m$ is invariant under $\mathrm{SO}(2n,\CC)$-conjugation, but not by  $\mathrm{O}(2n,\CC)$-conjugation. Furthermore (see \cite[Proposition 10]{sikora_generating_2013}), $Q_{2n}$ is the full polarization of the Pfaffian, so $Q_{2n}(A,\dots,A) = 2^n n! \mathrm{Pf}(A - {}^{t}\!A)$.

Now, let $\gamma_1,\dots,\gamma_n \in \Gamma$. Since the function $Q_{2n}$ is invariant under $\mathrm{SO}(2n,\CC)$-conjugation, it induces an invariant function $Q_{\gamma_1,\dots,\gamma_n}$ on $\mathrm{Hom}(\Gamma,\mathrm{SO}(2n,\CC))$ defined by
\[Q_{\gamma_1,\dots,\gamma_n}(\rho) = Q_{2n}(\rho(\gamma_1),\dots,\rho(\gamma_n)).\]

The following theorem, due to Sikora, gives a generating set for the invariant functions in the $\mathrm{SO}(2m,\CC)$ case.

\begin{thm}(\cite[Theorem 8]{sikora_generating_2013})
	
	The ring $\CC[ \mathrm{Hom}(\Gamma,\mathrm{SO}(2n,\CC))]^{\mathrm{SO}(2n,\CC)}$ is generated by the trace functions and by the functions $Q_{\gamma_1,\dots,\gamma_n}$ for $\gamma_1,\dots,\gamma_n \in \Gamma$.
\end{thm}

As a consequence of the previous theorem, we can use the trace functions and the functions $Q_{\gamma_1,\dots,\gamma_n}$ as coordinates for the character variety $\mathcal{X}_{\mathrm{SO(2n,\CC)}}(\Gamma)$.

\subsection{Maps between character varieties}

By the previous sections, we know that if $G_1$ and $G_2$ are two groups among $\mathrm{GL}(n,\CC)$, $\mathrm{SL}(n,\CC)$, $\mathrm{O}(n,\CC)$, $\mathrm{Sp}(2n,\CC)$ or $\mathrm{SO}(2n+1,\CC)$, the rings of invariant functions $\CC[\mathrm{Hom}(\Gamma,G_1)]^{G_1}$ and $\CC[\mathrm{Hom}(\Gamma,G_2)]^{G_2}$ are generated by the trace functions.
Hence, if $G_1 \subset G_2$ is a minimal inclusion, the inclusion $\mathrm{Hom}(\Gamma,G_1) \hookrightarrow \mathrm{Hom}(\Gamma,G_2)$ induces a surjective map $\CC[\mathrm{Hom}(\Gamma,G_2)]^{G_2} \twoheadrightarrow \CC[\mathrm{Hom}(\Gamma,G_1)]^{G_1}$.
Therefore, by considering the spectrum, we obtain the embedding of schemes $\mathfrak{X}_{G_1}(\Gamma) \hookrightarrow \mathfrak{X}_{G_2}(\Gamma)$.

Consider first the case of odd dimension, and the subgroups of $\mathrm{GL}(2n+1,\CC)$. In this case, the ring of invariant functions is generated by trace functions for all the classical subgroups. We have the following commutative diagram of representations:

\[
\begin{array}{ccc}
\mathrm{Hom}(\Gamma,\mathrm{SL}(2n+1,\CC)) & \hookrightarrow & \mathrm{Hom}(\Gamma,\mathrm{GL}(2n+1,\CC)) \\
\hookuparrow & & \hookuparrow \\
\mathrm{Hom}(\Gamma,\mathrm{SO}(2n+1,\CC)) & \hookrightarrow & \mathrm{Hom}(\Gamma,\mathrm{O}(2n+1,\CC))
\end{array}
\]

Therefore, it induces the following commutative diagram of character varieties, where all the arrows are embeddings:

\[
\begin{array}{ccc}
\mathcal{X}_{\mathrm{SL}(2n+1,\CC)}(\Gamma) & \hookrightarrow & \mathcal{X}_{\mathrm{GL}(2n+1,\CC)}(\Gamma) \\
\hookuparrow & & \hookuparrow \\
\mathcal{X}_{\mathrm{SO}(2n+1,\CC)}(\Gamma) & \hookrightarrow & \mathcal{X}_{\mathrm{O}(2n+1,\CC)}(\Gamma)
\end{array}
\]

The situation is the same in even dimension, except from the case $\mathrm{SO}(2n,\CC)$, for which the ring of invariant functions is not generated by trace functions. However, the trace functions are invariant by $\mathrm{SO}(2n,\CC)$-conjugation, so there are natural maps $\CC[\mathrm{Hom}(\Gamma,\mathrm{SL}(2n,\CC))]^{\mathrm{SL}(2n,\CC)} \rightarrow \CC[\mathrm{Hom}(\Gamma,\mathrm{SO}(2n,\CC))]^{\mathrm{SO}(2n,\CC)}$ and $\CC[\mathrm{Hom}(\Gamma,\mathrm{O}(2n,\CC))]^{\mathrm{O}(2n,\CC)} \rightarrow \CC[\mathrm{Hom}(\Gamma,\mathrm{SO}(2n,\CC))]^{\mathrm{SO}(2n,\CC)}$.
The situation of the representation varieties is described by the following commutative diagram:
\[
\begin{array}{ccccc}
\mathrm{Hom}(\Gamma,\mathrm{Sp}(2n,\CC)) & \hookrightarrow & \mathrm{Hom}(\Gamma,\mathrm{SL}(2n,\CC)) & \hookrightarrow & \mathrm{Hom}(\Gamma,\mathrm{GL}(2n,\CC)) \\
& & \hookuparrow & & \hookuparrow \\
& & \mathrm{Hom}(\Gamma,\mathrm{SO}(2n,\CC)) & \hookrightarrow & \mathrm{Hom}(\Gamma,\mathrm{O}(2n,\CC))
\end{array}
\]

Therefore, we have the following diagram for the corresponding character varieties, where all the arrows are embeddings, except from the ones with source $\mathcal{X}_{\mathrm{SO}(2n,\CC)}(\Gamma)$.

\[
\begin{array}{ccccc}
\mathcal{X}_{\mathrm{Sp}(2n,\CC)}(\Gamma) & \hookrightarrow & \mathcal{X}_{\mathrm{SL}(2n,\CC)}(\Gamma) & \hookrightarrow & \mathcal{X}_{\mathrm{GL}(2n,\CC)}(\Gamma) \\
& & \uparrow & & \hookuparrow \\
& & \mathcal{X}_{\mathrm{SO}(2n,\CC)}(\Gamma) & \rightarrow & \mathcal{X}_{\mathrm{O}(2n,\CC)}(\Gamma)
\end{array}
\]

The two maps with source $\mathcal{X}_{\mathrm{SO}(2n,\CC)}(\Gamma)$ are at most $2:1$. Indeed, if $\rho_1, \rho_2 \in \mathrm{Hom}(\Gamma,\mathrm{SO}(2n,\CC))$ are semi-simple representations with the same image in $\mathcal{X}_{\mathrm{O}(2n,\CC)}(\Gamma)$, then they are conjugated in $\mathrm{O}(2n,\CC)$. Consider a fixed matrix $M \in \mathrm{O}(2n,\CC) \setminus \mathrm{SO}(2n,\CC)$ and let $P \in \mathrm{O}(2n,\CC)$ such that for all $\gamma \in \Gamma$, $P\rho_1(\gamma)P^{-1} = \rho_2(\gamma)$. Since $\mathrm{SO}(2n,\CC)$ is of index two in $\mathrm{O}(2n,\CC)$, we have two cases. If $P \in \mathrm{SO}(2n,\CC)$, then $\rho_1$ and $\rho_2$ define the same point in $\mathcal{X}_{\mathrm{SO}(2n,\CC)}(\Gamma)$. If $P \notin \mathrm{SO}(2n,\CC)$, then $PM \in \mathrm{SO}(2n,\CC)$ and $\rho_2$ is conjugated to the representation $\gamma \mapsto M\rho_1(\gamma)M^{-1}$, and thus defines the same point in $\mathcal{X}_{\mathrm{SO}(2n,\CC)}(\Gamma)$. Hence the arrow $\mathcal{X}_{\mathrm{SO}(2n,\CC)} \to \mathcal{X}_{\mathrm{O}(2n,\CC)}$ is at most $2:1$.

Since the diagram is commutative and the two other arrows are embeddings, we deduce that the arrow $\mathcal{X}_{\mathrm{SO}(2n,\CC)} \to \mathcal{X}_{\mathrm{SL}(2n,\CC)}$ is at most $2:1$ as well.

\subsection{Definition of character varieties for real forms and anti-holomorphic involutions}

\subsubsection{Character varieties for real forms}
In \cite{acosta_character_2019}, we defined the character variety of  a finitely generated group $\Gamma$ for a real form $G$ of $\mathrm{SL}(n,\CC)$ as the image of $\mathrm{Hom}(\Gamma,G)$ by the projection $\mathrm{Hom}(\Gamma,\mathrm{SL}(n,\CC)) \twoheadrightarrow \mathcal{X}_{\mathrm{SL}(n,\CC)}(\Gamma)$. Following the same idea, we give a definition for the character variety for a real form of a classical complex group. Observe that, by definition, $\mathcal{X}_{G_\RR}(\Gamma)$ is embedded into $\mathcal{X}_{G_\CC}(\Gamma)$.

\begin{defn}
	Let $n$ be a positive integer, $\Gamma$ a finitely generated group and $G_{\CC}$ be one of the groups $\mathrm{GL}(n,\CC)$, $\mathrm{SL}(n,\CC)$, $\mathrm{O}(n,\CC)$, $\mathrm{SO}(n,\CC)$ or $\mathrm{Sp}(2n,\CC)$. Let $G_{\RR}$ be a real form of $G_\CC$. We define the \emph{$G_\RR$-character variety of $\Gamma$}, denoted $\mathcal{X}_{G_\RR}(\Gamma)$, as the image of $\mathrm{Hom}(\Gamma,G_\RR)$ by the projection $\mathrm{Hom}(\Gamma,G_\CC) \twoheadrightarrow \mathcal{X}_{G_\CC}(\Gamma)$.
\end{defn}

Now, we consider the anti-holomorphic involutions of the different character varieties induced by the involutions $A \mapsto \conjug{A}$ and $A \mapsto \transp{\conjug{A}}^{-1}$ of $\mathrm{GL}(n,\CC)$. We will have to deal with two different cases; depending on the fact that the invariant functions are generated by trace functions or not. We observe that the character varieties for real forms, as defined above, lie in the fixed points of these anti-holomorphic involutions.

\subsubsection{When the ring of functions is generated by trace functions}

We consider here the first case, which holds if $G_\CC$ is one of the groups $\mathrm{GL}(n,\CC)$, $\mathrm{SL}(n,\CC)$, $\mathrm{O}(n,\CC)$, $\mathrm{SO}(2n+1,\CC)$ or $\mathrm{Sp}(2n,\CC)$. In this case, the ring of invariant functions of $\mathrm{Hom}(\Gamma,G_\CC)$ is generated by trace functions.

\begin{prop}
	Let $G_\CC$ be one of the groups $\mathrm{GL}(n,\CC)$, $\mathrm{SL}(n,\CC)$, $\mathrm{O}(n,\CC)$, $\mathrm{SO}(2n+1,\CC)$ or $\mathrm{Sp}(2n,\CC)$. Then:
	\begin{enumerate}
		\item The involution $A \mapsto \conjug{A}$ induces an anti-holomorphic involution $\Phi_1$ on $\mathcal{X}_{G_\CC}(\Gamma)$.
		\item The involution $A \mapsto \transp{\conjug{A}}^{-1}$ induces an anti-holomorphic involution $\Phi_2$ on $\mathcal{X}_{G_\CC}(\Gamma)$.
		\item If $G_\CC$ is one of the groups  $\mathrm{O}(n,\CC)$, $\mathrm{SO}(2n+1,\CC)$ or $\mathrm{Sp}(2n,\CC)$, then $\Phi_1 = \Phi_2$.
	\end{enumerate}
\end{prop}

\begin{proof}
	Let us begin by proving the first two assertions.
	We know that the ring of invariant functions of $\mathrm{Hom}(\Gamma,G_\CC)$ is generated by a finite number of trace functions. Hence, there exist $\gamma_1,\dots,\gamma_s \in \Gamma $ such that $\mathcal{X}_{G_\CC}(\Gamma)$ is isomorphic to the image of the map
	\[
	\Psi \colon 
	\begin{array}{rcl}
	\mathrm{Hom}(\Gamma,G_\CC) &\to& \CC^s \\
	\rho &\mapsto& (\mathrm{tr}(\rho(\gamma_1)), \dots , \mathrm{tr}(\rho(\gamma_s)) )
	\end{array}
	\]
	
	By considering a larger set, we can suppose that $s=2t$, that $\{\gamma_1 , \dots , \gamma_{2t}\}$ generates $\Gamma$ and for all $i \in \{1,\dots,t\}$ we have $\gamma_{i+t} = \gamma_i^{-1}$.
	
	Consider a point $\chi \in \mathcal{X}_{G_\CC}(\Gamma)$ and a representation $\rho \in \mathrm{Hom}(\Gamma,G_\CC)$ such that $\chi = \chi_\rho$.
	Since the groups $\mathrm{O}(n,\CC)$, $\mathrm{SO}(2n+1,\CC)$ and $\mathrm{Sp}(2n,\CC)$ are stable by the involutions $A \mapsto \conjug{A}$ and $A \mapsto \transp{\conjug{A}}^{-1}$, the representations $\rho_1$ and $\rho_2$ given by 
	$\rho_1(\gamma) = \conjug{\rho(\gamma)}$ and $\rho_2(\gamma) = \transp{\conjug{\rho(\gamma)}}^{-1}$ belong to $\mathrm{Hom}(\Gamma,G_\CC)$.
	Since for all $i \in \{1,\dots,t\}$ we have $\gamma_{i+t} = \gamma_i^{-1}$, if $\Psi(\rho) = (z_1,\dots , z_t , w_1 , \dots , w_t)$, then we have
	\begin{align*}
		\Psi(\rho_1) &= (\conjug{z_1},\dots , \conjug{z_t} , \conjug{w_1} , \dots , \conjug{w_t}) \\
		\Psi(\rho_2) &= ( \conjug{w_1} , \dots , \conjug{w_t},\conjug{z_1},\dots , \conjug{z_t}).
	\end{align*}
	Therefore, the maps
	\begin{align*}
		(z_1,\dots,z_t , w_1, \dots w_t) &\mapsto (\conjug{z_1},\dots , \conjug{z_t} , \conjug{w_1} , \dots , \conjug{w_t}) \\
		(z_1,\dots,z_t , w_1, \dots w_t) &\mapsto ( \conjug{w_1} , \dots , \conjug{w_t},\conjug{z_1},\dots , \conjug{z_t}).
	\end{align*}
	define two anti-holomorphic involutions $\Phi_1$ and $\Phi_2$ on $\mathcal{X}_{G_\CC}(\Gamma)$.
	
	Let us prove now the third assertion. If $G_\CC$ is one of the groups  $\mathrm{O}(n,\CC)$, $\mathrm{SO}(2n+1,\CC)$ or $\mathrm{Sp}(2n,\CC)$, then there is bilinear form $B$ preserved by $G_\CC$. Thus, if $A \in G_\CC$, we have $\transp{A}BA = B$, and $\conjug{B} \conjug{A} \conjug{B}^{-1} = \transp{\conjug{A}}^{-1}$, so $\conjug{A}$ and $\transp{\conjug{A}}^{-1}$ have the same trace. Hence, $\Phi_1 = \Phi_2$.
	
\end{proof}

We will keep the notation of the above proposition for the involutions $\Phi_1$ and $\Phi_2$. If the group $G_\CC$ is one of the groups  $\mathrm{O}(n,\CC)$, $\mathrm{SO}(2n+1,\CC)$ or $\mathrm{Sp}(2n,\CC)$, the two involutions are equal, and $\Phi$ denote them. 
Let $\mathcal{X}_{G_\CC}^{\Phi_1}(\Gamma)$ (respectively $\mathcal{X}_{G_\CC}^{\Phi_2}(\Gamma)$ and $\mathcal{X}_{G_\CC}^{\Phi}(\Gamma)$) be the set of fixed points of $\Phi_1$ (respectively $\Phi_2$ and $\Phi$) in $\mathcal{X}_{G_\CC}(\Gamma)$.

\begin{rem}
	We have the following inclusions:
	\begin{align*}
		\mathcal{X}_{\mathrm{GL}(n,\RR)}(\Gamma) & \subset \mathcal{X}_{\mathrm{GL}(n,\CC)}^{\Phi_1}(\Gamma) 
		&
		\mathcal{X}_{\mathrm{GL}(n,\HH)}(\Gamma) & \subset \mathcal{X}_{\mathrm{GL}(2n,\CC)}^{\Phi_1}(\Gamma) &
		\mathcal{X}_{\mathrm{U}(p,q)}(\Gamma) & \subset \mathcal{X}_{\mathrm{GL}(n,\CC)}^{\Phi_2}(\Gamma) \\
		\mathcal{X}_{\mathrm{SL}(n,\RR)}(\Gamma) & \subset \mathcal{X}_{\mathrm{SL}(n,\CC)}^{\Phi_1}(\Gamma) 
		&
		\mathcal{X}_{\mathrm{SL}(n,\HH)}(\Gamma) & \subset \mathcal{X}_{\mathrm{SL}(2n,\CC)}^{\Phi_1}(\Gamma) 
		&
		\mathcal{X}_{\mathrm{SU}(p,q)}(\Gamma) & \subset \mathcal{X}_{\mathrm{SL}(n,\CC)}^{\Phi_2}(\Gamma) \\
		\mathcal{X}_{\mathrm{Sp}(2n,\RR)}(\Gamma) &\subset  \mathcal{X}_{\mathrm{Sp}(2n,\CC)}^{\Phi}(\Gamma) 
		&
		\mathcal{X}_{\mathrm{Sp}(2p,2q)}(\Gamma) &\subset  \mathcal{X}_{\mathrm{Sp}(2n,\CC)}^{\Phi}(\Gamma) 
		&
		\mathcal{X}_{\mathrm{O}(p,q)}(\Gamma) & \subset \mathcal{X}_{\mathrm{O}(n,\CC)}^{\Phi}(\Gamma) \\
		\mathcal{X}_{\mathrm{O}(n,\HH)}(\Gamma) & \subset \mathcal{X}_{\mathrm{O}(2n,\CC)}^{\Phi}(\Gamma) 
		&
		\mathcal{X}_{\mathrm{SO}(p,q)}(\Gamma) & \subset \mathcal{X}_{\mathrm{SO}(n,\CC)}^{\Phi}(\Gamma) \text{ if $n$ is odd} 
	\end{align*}
\end{rem}

\subsubsection{The case SO(2n,C)}
	First, observe that in the group $\mathrm{SO}(2n,\CC)$, the involutions $A \mapsto \conjug{A}$ and $A \mapsto {}^{t}\!\conjug{A}^{-1}$ are equal, so we only need to consider one of them for the computations. 

\begin{prop}
	The involution $A \mapsto \conjug{A}$ induces an anti-holomorphic involution $\Phi$ on $\mathcal{X}_{\mathrm{SO}(2n,\CC)}(\Gamma)$.
\end{prop}
\begin{proof}
	We know that the ring of invariant functions of $\mathrm{Hom}(\Gamma, \mathrm{SO}(2n,\CC) )$ is generated by a finite number of trace functions and of functions of the form $Q_{\gamma_1,\dots,\gamma_n}$. Let $\gamma_1, \dots , \gamma_m \in \Gamma$ and $c_1,\dots , c_s \in \Gamma^n$ such that $\tau_{\gamma_1}, \dots , \tau_{\gamma_m} , Q_{c_1},\dots , Q_{c_s}$ generates the ring of invariant functions.
	Hence, $\mathcal{X}_{\mathrm{SO}(2n,\CC)}(\Gamma)$ is isomorphic to th image of the map
	\[
	\Psi \colon 
	\begin{array}{rcl}
	\mathrm{Hom}(\Gamma,\mathrm{SO}(2n,\CC)) &\to& \CC^{m+s} \\
	\rho &\mapsto& (\tau_{\gamma_1}(\rho), \dots , \tau_{\gamma_m}(\rho) , Q_{c_1}(\rho), \dots Q_{c_s}(\rho) )
	\end{array}
	\]
	
	Let $\chi \in \mathcal{X}_{\mathrm{SO}(2n,\CC)}(\Gamma)$ and $\rho \in \mathrm{Hom}(\Gamma,\mathrm{SO}(2n,\CC))$ such that $\chi = \chi_\rho$. Since $\mathrm{SO}(2n,\CC)$ is stable by $A \mapsto \conjug{A}$, the representation $\conjug{\rho}$ defined by $\conjug{\rho}(\gamma)$ takes values in $\mathrm{SO}(2n,\CC)$.
	Since the trace and the function $Q_{2n}$ commute with the complex conjugation, if $\Psi(\rho) = (z_1, \dots , z_m , w_1, \dots , w_s)$, then $\Psi(\conjug{\rho}) = (\conjug{z_1}, \dots , \conjug{z_m} , \conjug{w_1}, \dots , \conjug{w_s})$. Therefore, the map
	\[
	(z_1, \dots , z_m , w_1, \dots , w_s)
	\mapsto
	(\conjug{z_1}, \dots , \conjug{z_m} , \conjug{w_1}, \dots , \conjug{w_s})
	\]
	induces an anti-holomorphic involution of $\mathcal{X}_{\mathrm{SO}(2n,\CC)}(\Gamma)$.
\end{proof}

Now, observe that we have the same inclusions of character varieties as before.
\begin{rem}
	We have the following inclusions:
	\begin{align*}
		\mathcal{X}_{\mathrm{SO}(p,q)}(\Gamma) & \subset \mathcal{X}_{\mathrm{SO}(2n,\CC)}^{\Phi}(\Gamma) 
		&
		\mathcal{X}_{\mathrm{SO}(n,\HH)}(\Gamma) & \subset \mathcal{X}_{\mathrm{SO}(2n,\CC)}^{\Phi}(\Gamma)
		&
		\mathcal{X}_{\mathrm{SO}^{-}(n,\HH)}(\Gamma) & \subset \mathcal{X}_{\mathrm{SO}(2n,\CC)}^{\Phi}(\Gamma)
	\end{align*}
\end{rem}

\section{Lifts of real irreducible characters}
In \cite{acosta_character_2019}, we proved the following result, by finding the appropriate conjugating matrices.
\begin{thm}(\cite[Theorem 1.1]{acosta_character_2019}) \label{thm:char_var_real_slnc}
	Let $\rho \in \mathrm{Hom}(\Gamma,\mathrm{SL}(n,\CC))$ be an irreducible representation. Then:
	\begin{enumerate}
		\item If $\chi_\rho \in \mathcal{X}^{\Phi_1}_{\mathrm{SL}(n,\CC)}(\Gamma)$, then $\rho$ is conjugated to a representation with values in $\mathrm{SL}(n,\RR)$ or to a representation with values in $\mathrm{SL}(n/2,\HH)$ (if $n$ is even).
		\item If $\chi_\rho \in \mathcal{X}^{\Phi_2}_{\mathrm{SL}(n,\CC)}(\Gamma)$, then $\rho$ is conjugated to a representation with values in a $\mathrm{SU}(p,q)$.
	\end{enumerate}
\end{thm}

 The proof provided in \cite{acosta_character_2019} works mutatis mutandis for $\mathrm{GL}(n,\CC)$-character varieties, so we will consider that the theorem above holds also in this case. In order to complete the proof of \Cref{thm:main_thm}, it only remains to prove similar statements for symplectic and orthogonal groups. Recall that, in these cases, we say that a representation is \emph{irreducible} if it is irreducible as a representation with values in $\mathrm{GL}(n,\CC)$. Observe that, in fact, we could consider \emph{good} representations as defined by Sikora in \cite[p. 5185]{sikora_character_2012a}, since the only consequence of irreducibility that is used is that the commutator of the image of an irreducible representation is the center of the group.
 
 For completeness, we give a general argument for proving a statement similar to \Cref{thm:main_thm}, but where a real form of a complex group is interpreted as an anti-holomorphic involution of the group.
 We will then give elementary proofs of the symplectic and orthogonal cases in Subsections \ref{subsect:symplectic} and \ref{subsect:orthogonal}, giving hints on how to find an explicit conjugation to a fixed real form.
 
 Let $G_\CC$ be a classical complex group and $\Phi$ be an anti-holomorphic $G_\CC$, inducing an involution $\Phi$ on the character variety $\mathcal{X}_{G_\CC}(\Gamma)$.
 If $\rho \in \mathrm{Hom}(\Gamma,G_\CC)$ is an irreducible representation such that $\chi_\rho \in \mathcal{X}^{\Phi}_{G_\CC}(\Gamma)$, then $\rho$ and $\Phi \circ \rho$ are irreducible representations that have the same character, and are therefore conjugate in $G_\CC$. The following proposition allows to conclude.
 
 \begin{prop}
 	Let $\Phi$ be an anti-holomorphic involution of a complex algebraic group $G_\CC$ defined over $\RR$.
 	Let $\rho \in \mathrm{Hom}(\Gamma, G_\CC)$ be an irreducible representation such that $\rho$ and $\Phi \circ \rho$ are conjugated in $G_\CC$. Then there is an anti-holomorphic involution of $G_\CC$ fixing the image of $\rho$.
 \end{prop}
 
 \begin{proof}
 	Let $P \in G_\CC$ such that for all $\gamma \in \Gamma$ we have $\rho(\gamma) = P \Phi(\rho(\gamma)) P^{-1}$. We have
 	\begin{align*}
 	\rho(\gamma) &= P \Phi(\rho(\gamma)) P^{-1} \\
 	\Phi(\rho(\gamma)) &= \Phi(P) \rho(\gamma) \Phi(P)^{-1} \\
 	P^{-1} \rho(\gamma) P &= \Phi(P) \rho(\gamma) \Phi(P)^{-1} \\
 	\rho(\gamma) &= P\Phi(P) \rho(\gamma) (P\Phi(P))^{-1}
 	\end{align*}
 	Thus, $P\Phi(P)$ commutes to the whole image of $\rho$, so, by Schur's lemma, it is a homothety.
 	Consider now the map
 	\[ \varphi \colon 
 	\begin{aligned}
 	G_\CC &\to G_\CC \\
 	A &\mapsto P \Phi(A) P^{-1}
 	\end{aligned}
 	\]
 	On the one hand, for all $\gamma \in \Gamma$ we have $\varphi(\rho(\gamma)) = \rho(\gamma)$. On the other hand, $\varphi$ is an involution, since $\varphi^2(A) = P\Phi(P) A (P\Phi(P))^{-1} = A$. Since $\varphi$ is anti-holomorphic, it is an anti-holomorphic involution of $G_\CC$ fixing the image of $\rho$.
 \end{proof}
 
\subsection{Real irreducible symplectic characters}\label{subsect:symplectic}

We deal first with the symplectic case. We begin by stating some preliminary results, namely that irreducible symplectic representations preserve at most one symplectic form, and that if two of them are conjugated in $\mathrm{SL}(2n , \CC)$ then they are conjugated in $\mathrm{Sp}(2n,\CC)$. Recall that we say, in this article, that a representation $\rho \in \mathrm{Hom}(\Gamma, \mathrm{Sp}(2n,\CC) )$ is irreducible if it is irreducible as a representation with values in $\mathrm{GL}(2n,\CC)$.

\begin{lemme}\label{lemma:bilinear_forms_preserved}
	Let $\rho \in \mathrm{Hom}(\Gamma,\mathrm{GL}(n,\CC))$ be an irreducible representation. If the image of $\rho$ preserves two non degenerate bilinear forms with matrices $J_1$ and $J_2$, then there exists $\lambda \in \CC^{*}$ such that $J_1 = \lambda J_2$. 
\end{lemme}

\begin{proof}
	Let $\gamma \in \Gamma$. Since $\rho(\gamma)$ preserves the two bilinear forms, we have $\transp{\rho(\gamma)} J_1 \rho(\gamma) = J_1$ and $\transp{\rho(\gamma)} J_2 \rho(\gamma) = J_2$. Hence,
	$J_1 \rho(\gamma) J_1^{-1} = \transp{\rho(\gamma)}^{-1} = J_2 \rho(\gamma) J_2^{-1} $. We deduce that $J_1^{-1}J_2$ commutes with $\rho(\gamma)$. Therefore $J_1^{-1}J_2$ commutes to the whole image of $\rho$. Since $\rho$ is irreducible, there exists $\lambda \in \CC^{*}$ such that $J_1 = \lambda J_2$.
\end{proof}

\begin{lemme}
	Let $\rho_1,\rho_2 \in \mathrm{Hom}(\Gamma, \mathrm{Sp}(2n,\CC))$ be two irreducible representations. If they are conjugate in $\mathrm{SL}(2n , \CC)$, then they are conjugate in $\mathrm{Sp}(2n,\CC)$.
\end{lemme}

\begin{proof}
	Let $M\in \mathrm{SL}(2n , \CC)$ such that for all $\gamma \in \Gamma$ we have $M\rho_1(\gamma)M^{-1} = \rho_2(\gamma)$. Thus, for all $\gamma \in \Gamma$,
	
	\begin{equation*}
		\transp{M}^{-1}\transp{\rho_1(\gamma)}\transp{M} J M \rho_1(\gamma) M^{-1} = J
	\end{equation*}
	
	Therefore, $\rho_1$ preserves the bilinear forms $J$ and $\transp{M}JM$. By \Cref{lemma:bilinear_forms_preserved}, there exists $\lambda \in \CC^*$ such that $\transp{M}JM = \lambda J$. Let $\mu$ be a square root of $\lambda$, and $M_1 = \frac{1}{\mu} M$. Then, $M_1$ conjugates $\rho_1$ and $\rho_2$, and $\transp{M_1}JM_1 = \frac{1}{\mu^2} \transp{M}JM= J$, i.e. $M \in \mathrm{Sp}(2n,\CC)$.
\end{proof}

In order to prove the main statement, we will use the following result of Fassbender and Ikramov:

\begin{prop} (\cite[Proposition 3]{fassbender_several_2005})\label{prop:hermitian+symplectic}
	Let $H \in \mathcal{M}_{2n}(\CC)$ be Hermitian and symplectic at the same time. Then, H admits a symplectic eigenvalue decomposition $H = VDV^*$	where the diagonal matrix $D$ and the unitary matrix $V$ are symplectic.
\end{prop}

\begin{cor}\label{cor:hermitian+symplectic}
	Let $H \in \mathcal{M}_{2n}(\CC)$ be Hermitian and symplectic at the same time. Then, there exist $p,q \in \NN$ and $S \in \mathrm{Sp}(2n,\CC)$ such that $\transp{\conjug{S}}HS = K_{p,q}$.
\end{cor}

\begin{proof}
	By \Cref{prop:hermitian+symplectic}, we can suppose that $H$ is symplectic and diagonal, hence of the form 
	$\begin{pmatrix}
	\Lambda & 0 \\
	0 & \Lambda^{-1}
	\end{pmatrix}$
	where $\Lambda	 \in \mathrm{GL}(n,\RR)$ is a diagonal matrix. Since $\Lambda$ is a real symmetric matrix, there exists $A \in \mathrm{GL}(n,\RR)$ and $p,q \in \NN$ such that $\transp{A}\Lambda A = I_{p,q}$. Then, the matrix $S = \begin{pmatrix}
	A & 0 \\
	0 & \transp{A}^{-1}
	\end{pmatrix}$ has real coefficients, is symplectic, and satisfies $\transp{\conjug{S}}HS = K_{p,q}$.
\end{proof}

By adapting the proof of Fassbender and Ikramov, we have the following proposition:

\begin{prop}\label{prop:hermitian+antisymplectic}
	Let $H \in \mathcal{M}_{2n}(\CC)$ be a Hermitian matrix such that $\transp{H}JH = -J$. Then, there exists a unitary matrix $V \in \mathrm{Sp}(2n,\CC)$ and a diagonal matrix $\Lambda \in \mathrm{GL}(n,\RR)$ with positive entries such that
	\[
	VHV^* = 
	\begin{pmatrix}
	\Lambda & 0 \\
	0 & - \Lambda^{-1}
	\end{pmatrix}
	\]
\end{prop}

\begin{proof}
	Let $\lambda$ be an eigenvalue of $H$ and $v$ be an associated unit eigenvector.
	Since we have $\transp{H}JH = -J$, it follows that $Jv = -\transp{H}JHv = -\lambda \transp{H}Jv$		or $\transp{H}(Jv)  = -\frac{1}{\lambda}(Jv)$. Considering the complex conjugation and using the fact that $H$ is Hermitian, we obtain
	\[H(J\conjug{v}) = -\frac{1}{\lambda}(J\conjug{v})\]
	
	Thus, $J\conjug{v}$ is a unit eigenvector of $H$ associated with the eigenvalue $-\frac{1}{\lambda}$. Since $H$ has real eigenvalues, $\lambda \neq -\frac{1}{\lambda}$, and the corresponding eigenvectors $v$ and $J\conjug{v}$ must be orthogonal.
	
	Let $\lambda_1 \geq \lambda_2 \geq \dots \geq \lambda_r > 0$ be the positive eigenvalues of $H$ and let $v_1,\dots,v_r$ be a corresponding orthonormal set of eigenvectors. 
	Then, $-\frac{1}{\lambda_1},-\frac{1}{\lambda_2}, \dots , -\frac{1}{\lambda_r}$ are the negative eigenvalues of $H$, and $J\conjug{v_1} , J\conjug{v_2},\dots , J\conjug{v_r} $ are the corresponding orthonormal eigenvectors. In particular, $r=n$ and $H$ has signature $(n,n)$.	Furthermore, the two sets of eigenvectors are orthogonal to each other since they correspond to non-overlapping sets of eigenvalues of the Hermitian matrix $H$.
	
	We construct the matrix $V$ by columns as:
	$V= (v_1, \cdots, v_n , -J\conjug{v_1}, \cdots , -J\conjug{v_n})$. It is clear that the matrix $V$ is unitary, and since the columns of $V$ are eigenvectors of $H$, we have
	\[
	VHV^* =
	\begin{pmatrix}
	\Lambda & 0 \\
	0 & -\Lambda^{-1}
	\end{pmatrix}
	\text{ with } \Lambda = 
	\begin{pmatrix}
	\lambda_1 & & \\
	& \ddots & \\
	& & \lambda_n
	\end{pmatrix}.
	\]
	
	It remains to verify that $V$ is symplectic. Let $W$ be the matrix by columns $W = (v_1, \cdots, v_n)$, so $V = (W,-J \conjug{W})$.
	
	We have:
	\[
	\transp{V}JV = 
	\begin{pmatrix}
	\transp{W} \\
	\transp{\conjug{W}J}
	\end{pmatrix}
	(JW, \conjug{W})
	=
	\begin{pmatrix}
	\transp{W}JW & \transp{W}\conjug{W}\\
	-\transp{\conjug{W}}W & \transp{\conjug{W}}J \conjug{W}
	\end{pmatrix}
	\]
	Since the columns of $W$ are orthonormal, we have $\transp{\conjug{W}}W = I_n$. Since the columns of $J\conjug{W}$ are orthogonal to the columns of $W$, we obtain $\transp{\conjug{W}}J \conjug{W} = 0$.
	By considering the complex conjugation we get $\transp{W}\conjug{W} = I_n$ and $\transp{W}J W = 0$, so $\transp{V}J V = J$ and $V$ is symplectic.
	
\end{proof}

\begin{cor}\label{cor:hermitian+antisymplectic}
	Let $H \in \mathcal{M}_{2n}(\CC)$ be a Hermitian matrix such that $\transp{H}JH = -J$. Then, there exist $p,q \in \NN$ and $S \in \mathrm{Sp}(2n,\CC)$ such that $\transp{\conjug{S}}HS = I_{n,n}$.
\end{cor}

\begin{proof}
	By \Cref{prop:hermitian+antisymplectic}, we can suppose that $H$ is diagonal of the form 
	$\begin{pmatrix}
	\Lambda & 0 \\
	0 & -\Lambda^{-1}
	\end{pmatrix}$
	where $\Lambda \in \mathrm{GL}(n,\RR)$ is a diagonal matrix with positive entries. Since $\Lambda$ is a positive definite real symmetric matrix, there exists $A \in \mathrm{GL}(n,\RR)$ such that $\transp{A}\Lambda A = I_n$. Then, the matrix $S = \begin{pmatrix}
	A & 0 \\
	0 & \transp{A}^{-1}
	\end{pmatrix}$ is real, symplectic, and satisfies $\transp{\conjug{S}}HS = I_{n,n}$.
\end{proof}

We now have all the tools to complete the proof of \Cref{thm:main_thm} in the symplectic case. The result is given the following proposition.

\begin{prop}
	Let $\rho \in \mathrm{Hom}(\Gamma, \mathrm{Sp}(2n, \CC))$ be an irreducible representation. Suppose that $\Phi(\chi_\rho) = \chi_\rho$. Then $\rho$ is conjugated in $\mathrm{Sp}(2n,\CC)$ to a representation with values in $\mathrm{Sp}(2n,\RR)$ or there exist $p,q \in \NN$ with $p+q = n$ such that $\rho$ is conjugated in $\mathrm{Sp}(2n,\CC)$ to a representation with values in $\mathrm{Sp}(2p,2q)$.
\end{prop}

\begin{proof}
	The character $\chi_\rho$ of the representation $\rho$ is a fixed point of the involution $\Phi$. Hence, $\chi_\rho$ is a fixed point of the involution $\Phi_2$ of $\mathcal{X}_{\mathrm{SL}(2n , \CC)}(\Gamma)$. By \Cref{thm:char_var_real_slnc}, the image of $\rho$ preserves a Hermitian form $H$. Hence, for all $\gamma \in \Gamma$, we have:
	\begin{align*}
		\transp{\rho(\gamma)} J \rho(\gamma) = J \\
		\transp{\conjug{\rho(\gamma)}} H \rho(\gamma) = H 
	\end{align*}
	Consider now the representation $\sigma \in \mathrm{Hom}(\Gamma, \mathrm{SL}(2n , \CC))$ given by
	\begin{equation*}
		\sigma(\gamma) = H\rho(\gamma)H^{-1} = \transp{\conjug{\rho{(\gamma)}}}^{-1}.
	\end{equation*}
	
	On the one hand, since the involution $A \mapsto \transp{\conjug{A}}^{-1}$ stabilizes $\mathrm{Sp}(2n,\CC)$, representation $\sigma$ takes values in $\mathrm{Sp}(2n,\CC)$. On the other hand, since $\sigma(\gamma) = H\rho(\gamma)H^{-1}$, the image of $\sigma$ preserves the bilinear form $\transp{H}^{-1} J H^{-1}$. By \Cref{lemma:bilinear_forms_preserved}, there exists $\lambda \in \CC$ such that $\transp{H}^{-1} J H^{-1} = \lambda J$. We can rewrite it as $\transp{H} J H J= -\lambda^{-1} I_{2n}$. Using the fact that $H$ is Hermitian, we deduce:
	\begin{align*}
		\transp{H} J H J= -\lambda^{-1} I_{2n} \\
		\transp{\conjug{H}} J \conjug{H} J= -\conjug{\lambda}^{-1} I_{2n} \\
		HJ \transp{H} J = -\conjug{\lambda}^{-1} I_{2n}
	\end{align*}
	
	Since $\mathrm{tr}(\transp{H} J H J) = \mathrm{tr}( H J \transp{H} J)$, we obtain that $-2n \lambda^{-1} = -2n \conjug{\lambda}^{-1} $, hence $\lambda \in \RR$. Perhaps after multiplying $H$ by a real scalar, we can suppose that $\lambda = \pm 1$. It remains to consider the two cases $\lambda = 1$ and $\lambda = -1$.

	\paragraph{First case: $\lambda = 1$.}
	We know that $\transp{H} J H =  J$. By the result of Fassbender and Ikramov in \Cref{prop:hermitian+symplectic} and \Cref{cor:hermitian+symplectic}, we know that there exist a symplectic  matrix $S$ and $p,q \in \NN$ such that $H = SK_{p,q}S^*$. Consider now the representation $\tau$ given by $ \tau(\gamma) = \transp{\conjug{S}}^{-1} \rho(\gamma)\transp{\conjug{S}}$. Since $S \in \mathrm{Sp}(2n,\CC)$, $\tau$ takes values in $\mathrm{Sp}(2n,\CC)$, and since it is conjugated to $\rho$, its image preserves the Hermitian form $SHS^* = K_{p,q}$. Therefore, $\tau \in \mathrm{Hom}(\Gamma,\mathrm{Sp}(2p,2q))$ and is conjugated to $\rho$ in $\mathrm{Sp}(2n,\CC)$.
	
	\paragraph{Second case: $\lambda = -1$.}
	We know then that $\transp{H} J H =  -J$. By \Cref{prop:hermitian+antisymplectic} and \Cref{cor:hermitian+antisymplectic}, there exists a symplectic matrix $S$ such that $\transp{\conjug{S}}HS = I_{n,n}$. In the same way, since $iJ$ is symplectic and  $\transp{(iJ)} J (iJ) = -J$, there exists a symplectic matrix $S_1$ such that $\transp{\conjug{S_1}}(iJ)S_1 = I_{n,n}$.
	
	Let $\tau$ be the representation defined by $\tau(\gamma) = (S_1S^{-1}) \sigma(\gamma) (S_1S^{-1})^{-1} $
	The representation $\tau$ is conjugated to $\rho$ in $\mathrm{Sp}(2n,\CC)$. Furthermore, since $\rho$ preserves the Hermitian form $H$, $\tau$ preserves the following Hermitian form:
	\begin{align*}
		\transp{\conjug{(S_1S^{-1})}}^{-1} H (S_1S^{-1})^{-1} 
		&= \transp{\conjug{S_1}}^{-1} (\transp{\conjug{S}} H S ) S_1^{-1}\\
		&= \transp{\conjug{S_1}}^{-1}I_{n,n}S_1^{-1} \\
		&= iJ
	\end{align*}
	
	Therefore, for $\gamma \in \Gamma$, we have:
	
	\begin{align*}
		\transp{\tau(\gamma)} J \tau(\gamma) &= J \\
		\transp{\conjug{\tau(\gamma)}} (iJ) \tau(\gamma) &= iJ
	\end{align*}
	
	Hence, $\tau(\gamma) = \conjug{\tau(\gamma)}$, i.e. $\tau$ takes values in $\mathrm{Sp}(2n,\RR)$.
	
\end{proof}

\begin{rem}
	In the proof, we used the results of conjugation in $\mathrm{SL}(2n,\CC)$ for the involution $\Phi_2$ of $\mathcal{X}_{\mathrm{SL}(2n,\CC)}(\Gamma)$. Since the involutions $\Phi_1$ and $\Phi_2$ are equal when restricted to $\mathrm{Sp}(2n,\CC)$, we can also consider the results for the involution $\Phi_1$. We know, by \Cref{thm:char_var_real_slnc}, that an irreducible representation with character fixed by $\Phi_1$ is conjugated in $\mathrm{SL}(2n,\CC)$ to a representation with values in $\mathrm{SL}(2n,\RR)$ or to a representation with values in $\mathrm{SL}(n,\HH)$. In the first case, we recover representations with values in $\mathrm{Sp}(2n,\RR)$, and in the second case we recover all the other real forms, namely $\mathrm{Sp}(2p,2q)$.
	
%
	
\end{rem}

\subsection{Real irreducible orthogonal characters}\label{subsect:orthogonal}

In this subsection, we complete the proof of \Cref{thm:main_thm} by considering the orthogonal and special orthogonal groups. We need to deal separately with odd and even orthogonal groups, since the real forms and the invariant rings of functions are essentially different in both cases. 
Recall that, in this article, we say that a representation $\rho \in \mathrm{Hom}(\Gamma, \mathrm{O}(n,\CC) )$ or $\mathrm{Hom}(\Gamma, \mathrm{SO}(n,\CC) )$ is irreducible if it is irreducible as a representation with values in $\mathrm{GL}(n,\CC)$. We begin by stating a preliminary lemma.
\begin{lemme}\label{lemma:conj_gl_conj_o}
	Let $\rho_1, \rho_2 \in \mathrm{Hom}(\Gamma, \mathrm{O}(n,\CC))$ be two irreducible representations. If they are conjugated in $\mathrm{GL}(n,\CC)$, then they are conjugated in $\mathrm{O}(n,\CC)$.
\end{lemme}

\begin{proof}
	Let $Q \in \mathrm{GL}(n,\CC)$ such that $\rho_2(\gamma) = Q\rho_1(\gamma)Q^{-1}$. Thus, the representation $\rho_2$ preserves the non-degenerate bilinear forms given by the matrices $I_{n}$ and $\transp{Q^{-1}}Q^{-1}$. By \Cref{lemma:bilinear_forms_preserved}, since $\rho_2$ is irreducible, there exists $\lambda \in \CC^{*}$ such that $\transp{Q^{-1}}Q^{-1} = \lambda I_{n}$. Let $\mu \in \CC$ be a square root of $\lambda$ and $Q_1 = \mu^{-1} Q$. Then, $\transp{Q_1^{-1}}Q_1^{-1} =  I_{n}$ and $Q_1$ still conjugates $\rho_1$ to $\rho_2$.
\end{proof}

\subsubsection{Odd orthogonal groups}
We are now able to prove \Cref{thm:main_thm} for odd orthogonal groups. The result for $\mathrm{O}(2n+1,\CC)$ is given in the following proposition. Since $-I_{2n+1}$ is in the center of $\mathrm{O}(2n+1,\CC)$ and has determinant $-1$, we deduce immediately the same statement for $\mathrm{SO}(2n+1,\CC)$.

\begin{prop}\label{prop:lifts_real_points_odd_o}
	Let $\rho \in \mathrm{Hom}(\Gamma, \mathrm{O}(2n+1, \CC))$ be an irreducible representation. Suppose that $\chi_\rho \in \mathcal{X}_{\mathrm{O}(2n+1,\CC)}^{\Phi}(\Gamma)$. Then there exist $p,q \in \NN$ with $p+q = 2n+1$ such that $\rho$ is conjugated in $\mathrm{O}(2n+1,\CC)$ to a representation with values in the copy $\mathcal{O}_{p,q}$ of $\mathrm{O}(p,q)$.
\end{prop}

\begin{proof}
	Since the irreducible representations $\rho$ and $\conjug{\rho}$ have the same character in $\mathcal{X}_{\mathrm{O}(2n+1,\CC)}(\Gamma)$, they have the same character in $\mathcal{X}_{\mathrm{GL}(2n+1,\CC)}(\Gamma)$. By \Cref{thm:char_var_real_slnc}, there exists $Q \in \mathrm{GL}(2n+1,\CC)$ such that for all $\gamma \in \Gamma$ we have $Q\rho(\gamma)Q^{-1} = \conjug{Q\rho(\gamma)Q^{-1}} $.

	Consider the representation $\sigma \in \mathrm{Hom}(\Gamma, \mathrm{GL}(2n+1 , \RR)) $ given by $\gamma \mapsto Q\rho(\gamma)Q^{-1}$. Since it is conjugated to $\rho$, it preserves the complex bilinear form with matrix $\transp{Q^{-1}}Q^{-1}$. Let $A,B \in \mathcal{M}_{2n+1}(\RR)$ be the symmetric matrices such that $\transp{Q^{-1}}Q^{-1} = A + iB$. By considering the real and the imaginary part of the identity $\transp{\sigma(\gamma)}(A+iB)\sigma(\gamma) = A+iB $, we obtain that for all $\gamma \in \Gamma$ we have $\transp{\sigma(\gamma)}A\sigma(\gamma) = A $ and $\transp{\sigma(\gamma)}B\sigma(\gamma) = B $. Therefore, for all $\lambda \in \RR$ and for all $\gamma \in \Gamma$ we have $\transp{\sigma(\gamma)}(A+\lambda B)\sigma(\gamma) = A+ \lambda B $.
	
	We know that $\det(A+\lambda B)$ is a polynomial in $\lambda$ that equals $1$ when $\lambda = i$. Hence, there exists $\lambda_0 \in \RR$ such that $A + \lambda_0 B$ is invertible. Let $C = A + \lambda_0 B$. It is the matrix of a non-degenerate real bilinear form preserved by $\sigma$. Thus, the image of $\sigma$ is contained in $\mathrm{O}(C)$. If $C$ is of signature $(p,q)$, there exists a matrix $R \in \mathrm{GL}(2n+1 , \RR)$ such that $C = \transp{R} I_{p,q} R$. Hence, the representation $\gamma \mapsto RQ\rho(\gamma)Q^{-1}R^{-1}$ takes values in $\mathrm{O}(p,q)$.
	
	Now, consider the representation $\rho'$ given by $\rho'(\gamma) = (D_{p,q}RQ) \rho(\gamma) (D_{p,q}RQ)^{-1}$. It is conjugated to $\rho$ in $\mathrm{GL}(2n+1 , \CC)$ and takes values in $\mathcal{O}_{p,q} \subset \mathrm{O}(2n+1, \CC)$. By \Cref{lemma:conj_gl_conj_o}, since $\rho$ is irreducible, the representations $\rho$ and $\rho'$ are conjugate in $\mathrm{O}(2n+1,\CC)$.
\end{proof}

\begin{cor}
	Let $\rho \in \mathrm{Hom}(\Gamma, \mathrm{SO}(2n+1, \CC))$ be an irreducible representation. Suppose that $\chi_\rho \in \mathcal{X}_{\mathrm{SO}(2n+1,\CC)}^{\Phi}(\Gamma)$. Then there exist $p,q \in \NN$ with $p+q = 2n+1$ such that $\rho$ is conjugated in $\mathrm{SO}(2n+1,\CC)$ to a representation with values in the copy $\mathcal{SO}_{p,q}$ of $\mathrm{SO}(p,q)$.
\end{cor}

\begin{proof}
	Since $\mathcal{X}_{\mathrm{SO}(2n+1,\CC)}$ is embedded in $\mathcal{X}_{\mathrm{O}(2n+1,\CC)}(\Gamma)$, we know that $\chi_\rho \in \mathcal{X}_{\mathrm{O}(2n+1,\CC)}^{\Phi}(\Gamma)$. By \Cref{prop:lifts_real_points_odd_o}, we know that there exist $p,q \in \NN$ and $Q \in \mathrm{O}(2n+1,\CC)$ such that for all $\gamma \in \Gamma$ we have $Q \rho(\gamma) Q^{-1} \in \mathcal{O}_{p,q}$. Since the determinant is invariant by conjugation, we know that $Q \rho(\gamma) Q^{-1} \in \mathcal{SO}_{p,q}$. It only remains to show that we can choose $Q \in \mathrm{SO}(2n+1,\CC)$. If $\det(Q)=1$, then $Q \in \mathrm{SO}(2n+1,\CC)$, if not, $\det(Q)=-1$ and $-Q \in \mathrm{SO}(2n+1,\CC)$ conjugates $\rho$ to a representation with values in $\mathcal{SO}_{p,q}$.
\end{proof}

\subsubsection{Even orthogonal groups}

Now, we prove the last part of \Cref{thm:main_thm}, namely the one concerning even orthogonal groups. We begin by proving some preliminary results.

\begin{lemme}\label{lemma:hermitian+antiorthogonal}
	Let $H \in \mathrm{GL}(2n , \CC)$ be a Hermitian matrix such that $\transp{H}H = -I_{2n}$. Then, there exists $M \in \mathrm{GL}(2n , \CC)$ such that $\transp{\conjug{M}}HM = I_{n,n}$ and 
	$\transp{M}M
	=
	\begin{pmatrix}
	0 & I_n \\
	I_n & 0
	\end{pmatrix}$.
\end{lemme}

\begin{proof}
	Let $\lambda \in \RR^*$ be an eigenvalue of $H$ and $v \in \CC^{2n}$ a corresponding eigenvector. We know that $Hv = \lambda v$. By taking the complex conjugation, we know that $\conjug{H} \conjug{v} = \lambda \conjug{v}$. Since $H$ is Hermitian and $\transp{H}H = -I_{2n}$, $\conjug{H} = -H^{-1}$ and we can re-write this last equality as $-\lambda^{-1}\conjug{v} = H \conjug{v} $. Thus, $\conjug{v}$ is an eigenvector of $H$ for the eigenvalue $-\lambda^{-1}$. Therefore, the spectrum of $H$ (counted with multiplicity) is stable under $\lambda \mapsto - \lambda^{-1}$. In particular, $H$ has signature $(n,n)$.
	
	Let $\lambda_1 \geq \lambda_2 \geq \cdots \geq \lambda_n > 0$ be the positive eigenvalues of $H$. The negative eigenvalues of $H$ are then $-\lambda_1^{-1},\dots,-\lambda_n^{-1}$. Let $v_1,\dots,v_n \in \CC^{2n}$ be a set of orthonormal vectors such that for $i \in \{1,\dots , n\}$ we have $Hv_i = \lambda_i v_i$. Thus, $\conjug{v_1},\dots,\conjug{v_n}$ is an orthonormal set of eigenvectors of $H$ for the eigenvalues $-\lambda_1^{-1},\dots,-\lambda_n^{-1}$.
	
	Since they are eigenvectors for two disjoints sets of eigenvalues of $H$, the sets $(v_1,\dots,v_n)$ and $(\conjug{v_1},\dots,\conjug{v_n})$ are orthogonal, and $(v_1,\dots,v_n,\conjug{v_1},\dots,\conjug{v_n})$ is an orthonormal basis of $\CC^{2n}$. Let $V \in \mathrm{GL}(2n , \CC)$ be the matrix with columns $v_1,\dots,v_n,\conjug{v_1},\dots,\conjug{v_n}$. By construction, $V$ is unitary and, if $\Lambda \in \mathrm{GL}(n,\RR)$ is the diagonal matrix with entries $(\lambda_1,\dots,\lambda_n)$ we have
	\[
	V^*HV =
	\begin{pmatrix}
	\Lambda & 0 \\
	0 & -\Lambda^{-1}
	\end{pmatrix}.
	\]
	
	We compute now $\transp{V}V$. The entries of this matrix are of the form 
	$\transp{v}w$ where $v$ and $w$ belong to the set $ \{ v_1,\dots,v_n,\conjug{v_1},\dots,\conjug{v_n} \}$. Since it is a set of orthonormal vectors, $\transp{v}w =0$ unless $v = \conjug{w}$, and in that case we have $\transp{\conjug{w}}w = 1$. Hence,
	\[
	\transp{V}V =
	\begin{pmatrix}
	0 & I_n \\
	I_n & 0
	\end{pmatrix}.
	\]
	
	Let $D\in \mathrm{GL}(n,\RR)$ be the diagonal matrix with entries $(\lambda_1)^{-\frac{1}{2}},\dots,(\lambda_n)^{-\frac{1}{2}}$, 
	$A
	=
	\begin{pmatrix}
	D & 0 \\
	0 & \transp{D}^{-1}
	\end{pmatrix}
	$
	and $M=VA$.
	On the one hand, we have 
	\begin{align*}
		\transp{\conjug{M}}HM &= \transp{\conjug{A}}\transp{\conjug{V}}HVA \\
		&= \transp{\conjug{A}}
		\begin{pmatrix}
			\Lambda & 0 \\
			0 & -\Lambda^{-1}
		\end{pmatrix} A \\
		&= I_{n,n}.
	\end{align*}
	On the other hand, the matrix $A$ is orthogonal for the symmetric bilinear form 
	$\begin{pmatrix}
	0 & I_n \\
	I_n & 0
	\end{pmatrix} $, hence 
	\begin{align*}
		\transp{M}M &= \transp{A}\transp{V}VA \\
		&= \transp{A}
		\begin{pmatrix}
			0 & I_n \\
			I_n & 0
		\end{pmatrix}
		A \\
		&=\begin{pmatrix}
			0 & I_n \\
			I_n & 0
		\end{pmatrix}.
	\end{align*}
	
\end{proof}

\begin{prop}\label{prop:hermitian+aniorthogonal}
	Let $H \in \mathrm{GL}(2n , \CC)$ be a Hermitian matrix such that $\transp{H}H = -I_{2n}$. Then, there exists $M \in \mathrm{O}(2n,\CC)$ such that $\transp{\conjug{M}}HM = iJ$.
\end{prop}

\begin{proof}
	On the one hand, by \Cref{lemma:hermitian+antiorthogonal}, we know that there exists $A \in \mathrm{GL}(2n , \CC)$ such that $\transp{\conjug{A}}HA = I_{n,n}$ and 
	$\transp{A}A
	=
	\begin{pmatrix}
	0 & I_n \\
	I_n & 0
	\end{pmatrix}$.
	On the other hand, since $iJ$ is a Hermitian matrix satisfying $\transp{(iJ)}(iJ) = -I_{2n}$, we can apply \Cref{lemma:hermitian+antiorthogonal} again. Thus, there exists $B \in \mathrm{GL}(2n , \CC)$ such that $\transp{\conjug{B}}(iJ)B = I_{n,n}$ and 
	$\transp{B}B
	=
	\begin{pmatrix}
	0 & I_n \\
	I_n & 0
	\end{pmatrix}$.
	Let $M = AB^{-1}$. We have:
	\begin{align*}
		\transp{\conjug{M}}HM &= \transp{\conjug{B}}^{-1}\transp{\conjug{A}}HAB^{-1}\\
		&= \transp{\conjug{B}}^{-1} I_{n,n}  B^{-1} \\
		&= iJ
	\end{align*}
	and also
	\begin{align*}
		\transp{M}M &= \transp{B^{-1}}\transp{A}AB^{-1} \\
		&=\transp{B^{-1}}
		\begin{pmatrix}
			0 & I_n \\
			I_n & 0
		\end{pmatrix}
		B^{-1} \\
		&= I_{2n}.
	\end{align*}
\end{proof}

We are now able to prove \Cref{thm:main_thm} for even orthogonal groups. The result is given in the following proposition.

\begin{prop}\label{prop:lifts_real_points_even_o}
	Let $\rho \in \mathrm{Hom}(\Gamma, \mathrm{O}(2n, \CC))$ be an irreducible representation such that $\chi_\rho \in \mathcal{X}_{\mathrm{O}(2n,\CC)}^{\Phi}(\Gamma)$. Then either $\rho$ is conjugated in $\mathrm{O}(2n,\CC)$ to a representation with values in $\mathrm{O}(n,\HH)$ or there exist $p,q \in \NN$ with $p+q = 2n$ such that $\rho$ is conjugated in $\mathrm{O}(2n,\CC)$ to a representation with values in in the copy $\mathcal{O}_{p,q}$ of $\mathrm{O}(p,q)$.
\end{prop}

\begin{proof}
	Since the irreducible representations $\rho$ and $\transp{\conjug{\rho}}^{-1}$ have the same character in $\mathcal{X}_{\mathrm{O}(2n,\CC)}(\Gamma)$, they have the same character in $\mathcal{X}_{\mathrm{GL}(2n,\CC)}(\Gamma)$. By \Cref{thm:char_var_real_slnc}, the image of $\rho$ preserves a non-degenerate Hermitian form $H$. We know that for all $\gamma \in \Gamma$, we have
	\begin{align*}
		\transp{\rho(\gamma)}\rho(\gamma) = I_{2n} \\
		\transp{\conjug{\rho(\gamma)}} H \rho(\gamma) = H
	\end{align*}
	Hence, $\transp{\conjug{\rho(\gamma)}}^{-1} = H \rho(\gamma) H^{-1}$. On the one hand, the representation $\gamma \mapsto \transp{\conjug{\rho(\gamma)}}^{-1}$ is irreducible and orthogonal, since $\mathrm{O}(2n,\CC)$ is stable by $A \mapsto \transp{\conjug{A}}^{-1}$. On the other hand, the representation $\gamma \mapsto H \rho(\gamma) H^{-1}$ preserves the bilinear form $\transp{H}H$. By \Cref{lemma:bilinear_forms_preserved}, there exists $\lambda \in \CC^*$ such that $\transp{H}H = \lambda I_{2n}$. Since $H$ is Hermitian, by conjugating this last equality and taking the trace, we obtain that $2n\lambda = 2n\conjug{\lambda}$, and so $\lambda$ is real. By multiplying $H$ by a real constant, we can suppose that $\lambda = \pm 1$. Thus, $\transp{H}H = \pm I_{2n}$ and we have two cases.
	
	\paragraph{First case: $\transp{H}H = I_{2n}$.}
	Since $H$ is Hermitian, we have $\conjug{H}H = I_{2n}$. By Lemma 3.15 of \cite{acosta_character_2019}\footnote{Or Hilbert's Theorem 90.} we know that there exists $Q \in \mathrm{GL}(2n,\CC)$ such that $H = \conjug{Q}^{-1}Q $. Therefore, for all $\gamma \in \Gamma$, we have $Q\rho(\gamma)Q^{-1} = \conjug{Q\rho(\gamma)Q^{-1}}$, i.e. $Q\rho(\gamma)Q^{-1} \in \mathrm{GL}(2n , \RR)$.
	The rest of the proof is exactly the same as the proof of \Cref{prop:lifts_real_points_odd_o}.

	\paragraph{Second case: $\transp{H}H = -I_{2n}$.}
	By \Cref{lemma:hermitian+antiorthogonal} and \Cref{prop:hermitian+aniorthogonal}, there exists $M \in \mathrm{O}(2n,\CC)$ such that $\transp{\conjug{M}}HM = iJ$. Consider the representation $\sigma \in \mathrm{O}(2n,\CC)$ given by $\sigma(\gamma) = M\rho(\gamma)M^{-1}$. Since it is conjugated to $\rho$, it preserves the Hermitian form $\transp{\conjug{M}}HM = iJ$. Therefore, if $\gamma \in \Gamma$, we have
	\begin{align*}
		\transp{\conjug{\sigma(\gamma)}} iJ \sigma(\gamma) &= iJ \\
		\transp{\conjug{\sigma(\gamma)}} J \sigma(\gamma) &= J \\
		\conjug{\sigma(\gamma)}^{-1} J \sigma(\gamma) &= J \\
		J \sigma(\gamma) &= \conjug{\sigma(\gamma)} J \\
	\end{align*}
	Hence, $\sigma(\gamma) \in \mathrm{GL}(n,\HH)$, and $\sigma \in \mathrm{Hom}(\Gamma, \mathrm{O}(n,\HH))$.
	
\end{proof}

Finally, we conclude with the proof for $\mathrm{SO}(2n,\CC)$. This completes the proof of \Cref{thm:main_thm}.

\begin{prop}
	Let $\rho \in \mathrm{Hom}(\Gamma, \mathrm{SO}(2n, \CC))$ be an irreducible representation such that $\chi_\rho \in \mathcal{X}_{\mathrm{SO}(2n,\CC)}^{\Phi}(\Gamma)$. Then either
	there exist $p,q \in \NN$ with $p+q = 2n$ such that $\rho$ is conjugated in $\mathrm{SO}(2n,\CC)$ to a representation with values in the copy $\mathcal{SO}_{p,q}$ of $\mathrm{SO}(p,q)$,
	$\rho$ is conjugated in $\mathrm{SO}(2n,\CC)$ to a representation with values in $\mathrm{SO}(n,\HH)$ or 
	$\rho$ is conjugated in $\mathrm{SO}(2n,\CC)$ to a representation with values in $\mathrm{SO}^{-}(n,\HH)$.
\end{prop}

\begin{proof}
	Let $\pi : \mathcal{X}_{\mathrm{SO}(2n,\CC)} \to \mathcal{X}_{\mathrm{O}(2n,\CC)}(\Gamma)$ be the projection map. Since $\chi_\rho \in \mathcal{X}_{\mathrm{SO}(2n,\CC)}^{\Phi}(\Gamma)$, we know that $\pi(\chi_\rho) \in \mathcal{X}_{\mathrm{O}(2n,\CC)}^{\Phi}(\Gamma)$. By \Cref{prop:lifts_real_points_even_o}, we know that either $\rho$ is conjugated in $\mathrm{O}(2n,\CC)$ to a representation with values in $\mathrm{O}(n,\HH)$ or there exist $p,q \in \NN$ with $p+q = 2n$ such that $\rho$ is conjugated in $\mathrm{O}(2n,\CC)$ to a representation with values in the copy $\mathcal{O}_{p,q}$ of $\mathrm{O}(p,q)$. 
	Since the determinant is invariant by conjugation, we know that there exists a real form $G_\RR$ of $\mathrm{SO}(2n,\CC)$ and $Q \in \mathrm{O}(2n,\CC)$ such that $Q \rho(\gamma) Q^{-1} \in G_\RR$. It only remains to show that we can choose $Q \in \mathrm{SO}(2n,\CC)$. If $\det(Q)=1$, then $Q \in \mathrm{SO}(2n,\CC)$ and we have nothing to do. 
	Now, suppose that $\det(Q)=-1$. We need to consider three cases.
	
	\paragraph{First case: $G_\RR = \mathcal{SO}(p,q)$.}  In this case, consider the matrix $I_{2n-1,1} \in \mathcal{O}(p,q) \setminus \mathcal{SO}(p,q)$. Since it has determinant $-1$ and stabilizes $\mathcal{SO}(p,q)$ by conjugation, we know that $I_{2n-1,1} Q \in \mathrm{SO}(2n,\CC)$ and $(I_{2n-1,1} Q) \rho(\gamma) (I_{2n-1,1} Q)^{-1} \in \mathcal{SO}(p,q)$ for all $\gamma \in \Gamma$.
	
	\paragraph{Second case: $G_\RR = \mathrm{SO}(2n,\HH)$ and $n=2m+1$ is odd.}  In this case, consider the matrix $K =
	\begin{pmatrix}
	0 & I_n \\
	I_n & 0
	\end{pmatrix}
	 \in \mathrm{O}(2n,\CC) \setminus \mathrm{SO}(2n,\CC)$. It has determinant $-1$ and stabilizes $\mathrm{SO}(2n,\HH)$ by conjugation, so $K Q \in \mathrm{SO}(2n,\CC)$ and $(K Q) \rho(\gamma) (K Q)^{-1} \in \mathrm{SO}(2n,\HH)$ for all $\gamma \in \Gamma$.
	 
	 \paragraph{Third case: $G_\RR = \mathrm{SO}(2n,\HH)$ and $n=2m$ is even.}  In this case, consider a matrix $P_0
	 \in \mathrm{O}(2n,\CC) \setminus \mathrm{SO}(2n,\CC)$ that conjugates $\mathrm{SO}(2n,\HH)$ and $\mathrm{SO}^{-}(2n,\HH)$. It has determinant $-1$, so $P_0 Q \in \mathrm{SO}(2n,\CC)$ and $(P_0 Q) \rho(\gamma) (P_0 Q)^{-1} \in \mathrm{SO}^{-}(2n,\HH)$ for all $\gamma \in \Gamma$.

\end{proof}

\section{Character varieties for the compact real forms}

 Recall that one of the motivations for the construction of character varieties is that in general, if $\Gamma$ is a finitely generated group and $G$ a topological group, the quotient topology on $\mathrm{Hom}(\Gamma,G)/G$ may not be Hausdorff. However, if $G$ is compact, $\mathrm{Hom}(\Gamma,G)/G$ is a compact space for the quotient topology.
 A question which arises naturally is if $K$ is a compact real form of a classical complex group, the topological quotient $\mathrm{Hom}(\Gamma,K)/K$ is homeomorphic to the character variety $\mathcal{X}_K(\Gamma)$ defined as a subset of the character variety of the complex group.
 In \cite{acosta_character_2019}, this fact is proven for the $\mathrm{SU}(n)$-character varieties. The exact same proof works for $\mathrm{U}(n)$-character varieties.
 
 We give here a general proof, where the main part is given in Remark 4.7 of the article of Florentino and Lawton \cite{florentino_character_2013}. The key point is the existence and uniqueness of the polar decomposition. Recall that for $M \in \mathrm{GL}(n,\CC)$, there is a unique pair $(U,H)$ with $U \in \mathrm{U}(n)$ and $H$ a positive definite Hermitian matrix such that $M=UH$.
 
 The same polar decomposition exists also for the classical groups. In that case, as a consequence of the uniqueness in the polar decomposition, the two factors are also in the classical group, as stated in the following remark.
 
 \begin{rem}\label{rem:polar_dec_subgroups}
 	Let $U \in \mathrm{U}(n)$, $H \in \mathrm{GL}(n,\CC)$ a positive definite Hermitian matrix and $\mathrm{G}_\CC$ be one of the groups $\mathrm{SL}(n,\CC), \mathrm{O}(n,\CC), \mathrm{SO}(n,\CC)$ and $\mathrm{Sp}(2n,\CC)$. If $UH \in \mathrm{G}_\CC$, then $U \in \mathrm{G}_\CC$ and $H \in \mathrm{G}_\CC$.
 \end{rem} 
 
%
%
%
%
 
  The following statement is the same as Remark 4.7 of \cite{florentino_character_2013}. We give the proof again for completeness.
 
 \begin{prop}\label{prop:conj_g_implies_conj_compact}
 	Let $\mathrm{G}_\CC$ be one of the groups $\mathrm{GL}(n,\CC)$, $\mathrm{SL}(n,\CC), \mathrm{O}(n,\CC), \mathrm{SO}(n,\CC)$ and $\mathrm{Sp}(2n,\CC)$, and $K$ its compact real form. Let $\rho,\rho' \in \mathrm{Hom}(\Gamma,K)$ be two representations conjugated in $G_\CC$. Then, $\rho$ and $\rho'$ are conjugated in $K$.
 \end{prop}
 
 \begin{proof}
 	We are going to prove that if $\rho$ and $\rho'$ are conjugated by a matrix $P$ having polar decomposition $UH$, then $U$ also conjugates $\rho$ and $\rho'$.
 	
 	Let $P \in G_\CC$ conjugating $\rho$ and $\rho'$. Let $U \in \mathrm{U}(n)$ and $H \in \mathrm{GL}(n,\CC)$ a positive definite Hermitian matrix such that $P=UH$. By Remark \ref{rem:polar_dec_subgroups}, we know that $U \in K$. Now, let $\gamma \in \Gamma$.
 	We know that $P\rho(\gamma)P^{-1} = \rho'(\gamma)$, so
 	$UH\rho(\gamma)H^{-1}U^{-1} = \rho'(\gamma)$. By manipulating this last equation, we obtain:
 	\[U\rho(\gamma) ( \rho(\gamma)^{-1} H \rho(\gamma) ) = \rho'(\gamma) U H\]
 	Since $\rho(\gamma)$ and $\rho'(\gamma)$ belong to $K$, we have that $U\rho(\gamma) \in K$, and that $\rho'(\gamma) U \in K$, and $\rho(\gamma)^{-1} H \rho(\gamma)$ and $H$ are positive definite Hermitian matrices. Since the polar decomposition is unique, we have $U\rho(\gamma) = \rho'(\gamma) U$ and $\rho(\gamma)^{-1} H \rho(\gamma) = H$. In particular, $U\rho(\gamma)U^{-1} = \rho'(\gamma) $.
 \end{proof}
 
 Using the proposition above, we are now able to prove that the topological quotient $\mathrm{Hom}(\Gamma,K)/K$ is homeomorphic to the character variety $\mathcal{X}_K(\Gamma)$.
 
 \begin{prop}\label{prop:character_compact}
	Let $\Gamma$ be a finitely generated group and let $\mathrm{G}_\CC$ be one of the groups $\mathrm{GL}(n,\CC)$, $\mathrm{SL}(n,\CC), \mathrm{O}(n,\CC), \mathrm{SO}(n,\CC)$ and $\mathrm{Sp}(2n,\CC)$, and $K$ its compact real form.
	Then $\mathcal{X}_K(\Gamma)$ is homeomorphic to the topological quotient $\mathrm{Hom}(\Gamma,K)/K$.
 \end{prop}
 \begin{proof}
 	First, observe that there is a natural continuous map from $\mathrm{Hom}(\Gamma,K)/K$ to $\mathcal{X}_K(\Gamma)$. By definition, $\mathcal{X}_K(\Gamma)$ is the image of the natural map $\mathrm{Hom}(\Gamma,K) \to \mathcal{X}_{G_\CC}(\Gamma)$. Furthermore, since two representations which are conjugated in $K$ are also conjugated in $G_\CC$, the map factors through $\mathrm{Hom}(\Gamma,K)/K$. Thus, we have a map $\varphi : \mathrm{Hom}(\Gamma,K)/K \to \mathcal{X}_K(\Gamma)$. The map $\varphi$ is continuous and surjective by definition. Since $\mathrm{Hom}(\Gamma,K)/K$ is compact and $\mathcal{X}_K(\Gamma)$ is Hausdorff as a subset of $\CC^N$, we only need to prove that $\varphi$ is injective in order to have a homeomorphism.
 	
 	Let $\rho,\rho' \in \mathrm{Hom}(\Gamma,K)$ such that $\chi_\rho = \chi_{\rho'}$ in $\mathcal{X}_K(\Gamma)$. Since $K$ is compact, $K \subset U(n)$ and $\rho$ and $\rho'$ are semisimple. Hence, $\rho$ and $\rho'$ are conjugated in $G_\CC$. By \Cref{prop:conj_g_implies_conj_compact}, $\rho$ and $\rho'$ are conjugated in $K$ and thus represent the same point of $\mathrm{Hom}(\Gamma,K)/K$. Hence, $\varphi$ is injective and a homeomorphism.
 \end{proof}

\section{Reducible representations with values in $\mathrm{GL}(n,\CC)$}

Until now, we have only considered irreducible representations for proving that points of $\mathcal{X}^{\Phi}_{G_\CC}(\Gamma)$ are in fact points of $\mathcal{X}_{G_\RR}(\Gamma)$ for some real form $G_\RR$ of $G_\CC$. In this last section, we are going to consider reducible representations as well, but only for $\mathrm{GL}(n,\CC)$-representations. We prove that a semi-simple representation whose character is fixed by $\Phi_2$ is conjugated to a representation with values in a $\mathrm{U}(p,q)$. This fact, corresponding to \Cref{prop:reducible_fixed_phi2}, can be stated in the following way:

\begin{thm}\label{thm:char_var_unitary}
	Let $\Gamma$ be a finitely generated group and $n$ be a positive integer. Then
	\begin{equation*}
		\mathcal{X}^{\Phi_2}_{\mathrm{GL}(n,\CC)}(\Gamma) = 
		\bigcup_{p+q = n} \mathcal{X}_{\mathrm{U}(p,q)}(\Gamma).
	\end{equation*}
\end{thm}

 We cannot expect a similar statement for the fixed points of $\Phi_1$, since the direct sum of $\rho \in \mathrm{Hom}(\Gamma,\mathrm{GL}(m_1 , \RR))$ and $\rho' \in \mathrm{Hom}(\Gamma,\mathrm{GL}(m_2 , \HH))$ will have a character fixed by $\Phi_1$ but its image is not fixed by any anti-holomorphic involution of $\mathrm{GL}(m_1+2m_2 , \CC)$. However, if $\rho \in \mathrm{Hom}(\Gamma, \mathrm{GL}(n,\CC))$ is semi-simple, we are able to prove that if $\chi_\rho \in \mathcal{X}^{\Phi_1}_{\mathrm{GL}(n,\CC)}(\Gamma)$, then $\rho$ is conjugated to the direct sum of a representation with values in some $\mathrm{GL}(m_1 , \RR)$ and a representation with values in some $\mathrm{GL}(m_2 , \HH)$. It is the content of \Cref{prop:reducible_fixed_phi1}.

 We begin by studying semi-simple representations as direct sums of irreducible representations. A capital fact is that the multiplicity of an irreducible representation is well defined, as stated in the following remark.
 
\begin{rem}
	The multiplicity of an irreducible representation $(\rho , V)$ in a semi-simple representation $(\rho',W)$ is well defined as
	$\dim (\mathrm{Hom}(V,W)^{\Gamma})$.
	The key point is the conclusion of Schur's lemma.
\end{rem}

 In the three following lemmas, we study the decomposition in irreducible factors of a semi-simple representation $\rho$ which is conjugated to $\Phi \circ \rho$.

\begin{lemme}\label{lemma:direct_sum_invariant}
	Let $\Phi = \Phi_1$ or $\Phi_2$ be an anti-holomorhphic involution of $\mathrm{GL}(n,\CC)$. Let $\rho \in \mathrm{Hom}(\Gamma, \mathrm{GL}(n,\CC))$ be a semi-simple representation such that $\rho$ and $\Phi \circ \rho$ are conjugated in $\mathrm{GL}(n,\CC)$. Then $\rho$ is conjugated in $\mathrm{GL}(n,\CC)$ to a representation of the form
	\begin{equation*}
		\bigoplus_{i=1}^r \rho_i^{\oplus m_i} \oplus \bigoplus_{i=r+1}^{s} (\rho_i \oplus \Phi(\rho_i))^{\oplus m_i}
	\end{equation*}
	where the $\rho_i$ are irreducible representations and $\rho_i$ is conjugated to $\Phi(\rho_i)$ for $i \in \{1,\dots, r\}$.
\end{lemme}

\begin{proof}
	Since $\rho$ is semi-simple, there exist $k \in \NN^*$, $m_1, \cdots , m_k \in \NN^*$ and distinct irreducible representations $\rho_1, \dots , \rho_k$ of $\Gamma$ such that
	\begin{equation*}
		\rho = \rho_1^{\oplus m_1} \oplus \cdots \oplus \rho_k^{\oplus m_k}
	\end{equation*}
	Since the representations $\rho$ and $\Phi(\rho)$ are conjugated in $\mathrm{GL}(n,\CC)$, they have the same irreducible representations as factors with the same multiplicity. However, 
	\begin{equation*}
		\Phi(\rho) = \Phi(\rho_1)^{\oplus m_1} \oplus \cdots \oplus \Phi(\rho_k)^{\oplus m_k}
	\end{equation*}
	so for each $i \in \{1, \dots , k \}$, either $\rho_i$ is conjugate to $\Phi(\rho_i)$, or there is some $j \in \{1, \dots , k\}$ such that $\rho_i$ is conjugate to $\Phi(\rho_j)$. In the second case, the representations $\rho_i$ and $\rho_j$ have the same multiplicity.
	Thus, after renumeration of the irreducible representations
	\begin{equation*}
	\rho =
		\bigoplus_{i=1}^r \rho_i^{\oplus m_i} \oplus \bigoplus_{i=r+1}^{s} (\rho_i \oplus \rho'_i)^{\oplus m_i}
	\end{equation*}
	Where $\rho_i$ is conjugate to $\Phi(\rho_i)$ for $i \in \{1 , \dots ,r  \}$ and $\rho'_i$ is conjugate to $\Phi(\rho_i)$ for $i \in \{r+1, \dots , s\}$. 
	Hence, $\rho$ is equivalent to
	\begin{equation*}
		\bigoplus_{i=1}^r \rho_i^{\oplus m_i} \oplus \bigoplus_{i=r+1}^{s} (\rho_i \oplus \Phi(\rho_i))^{\oplus m_i}.
	\end{equation*}
\end{proof}

\begin{lemme}\label{lemma:rho_plus_rho_conj}
	Let $\rho \in \mathrm{Hom}(\Gamma, \mathrm{GL}(n,\CC))$ be any representation. Then, the representation $\rho \oplus \Phi_1(\rho)$ is conjugated to a representation with values in $\mathrm{GL}(n,\HH)$.
\end{lemme}
\begin{proof}
	Consider the representation $\sigma$ defined by
	\begin{align*}
		\sigma(\gamma) &= 
		\begin{pmatrix}
		I_n & I_n \\
		-I_n & I_n
		\end{pmatrix}
		\begin{pmatrix}
		\rho(\gamma) & 0 \\
		0 & \conjug{\rho(\gamma)}
		\end{pmatrix}
		\begin{pmatrix}
		I_n & I_n \\
		-I_n & I_n
		\end{pmatrix}^{-1} \\
		&=
		\frac{1}{2}
		\begin{pmatrix}
		\rho(\gamma)+\conjug{\rho(\gamma)} & \rho(\gamma)-\conjug{\rho(\gamma)} \\
		\rho(\gamma)-\conjug{\rho(\gamma)} & \rho(\gamma)+\conjug{\rho(\gamma)}
		\end{pmatrix}
	\end{align*}
	The representation $\sigma$ is conjugated to $\rho \oplus \conjug{\rho}$ and for all $\gamma \in \Gamma$ we have $\sigma(\gamma) J_{2n} = J_{2n} \conjug{\sigma(\gamma)}$. Hence $\sigma \in \mathrm{Hom}(\Gamma, \mathrm{GL}(n,\HH))$.
\end{proof}

\begin{lemme}\label{lemma:rho_plus_rho_phi2}
	Let $\rho \in \mathrm{Hom}(\Gamma, \mathrm{GL}(n,\CC))$ be any representation. Then, the representation $\rho \oplus \Phi_2(\rho)$ is unitary for the Hermitian form with matrix
	$
	\begin{pmatrix}
		0 & I_n \\
		I_n & 0
	\end{pmatrix}
	$.
\end{lemme}
\begin{proof}
	Let $\gamma \in \Gamma$. We have:
	\begin{align*}
		\transp{
		\conjug{
		\begin{pmatrix}
		\rho(\gamma) & 0 \\
		0 & {}^t\!\conjug{\rho(\gamma)}^{-1}	
		\end{pmatrix}
		}}
		\begin{pmatrix}
			0 & I_n \\
			I_n & 0
		\end{pmatrix}
		\begin{pmatrix}
			\rho(\gamma) & 0 \\
			0 & {}^t\!\conjug{\rho(\gamma)}^{-1}	
		\end{pmatrix}
		&=
		\begin{pmatrix}
			0 & {}^t\!\conjug{\rho(\gamma)} {}^t\!\conjug{\rho(\gamma)}^{-1} \\
			\rho(\gamma)^{-1}\rho(\gamma) & 0
		\end{pmatrix} \\
		&=
		\begin{pmatrix}
			0 & I_n \\
			I_n & 0
		\end{pmatrix}
	\end{align*}
\end{proof}

With these results, we are able to prove that semi-simple representations with characters fixed by an anti-holomorphic involution $\Phi$ of $\mathcal{X}_{\mathrm{GL}(n,\CC)}(\Gamma)$ are conjugated to direct sums of representations with values in real forms of complex linear groups. We begin with the involution $\Phi_1$, related to $\mathrm{GL}(m,\RR)$ and $\mathrm{GL}(m,\HH)$ groups.

\begin{prop}\label{prop:reducible_fixed_phi1}
	Let $\rho \in \mathrm{Hom}(\Gamma,\mathrm{GL}(n,\CC))$ be a semi-simple representation. If $\Phi_1(\chi_\rho) = \chi_\rho$, then there exist $r_1,r_2 \in \NN$ with $r_1+2r_2 = n$, $\rho' \in \mathrm{Hom}(\Gamma, \mathrm{GL}(r_1,\RR))$ and $\rho'' \in \mathrm{Hom}(\Gamma, \mathrm{GL}(r_2,\HH))$ such that $\rho$ is conjugated to $\rho' \oplus \rho''$.
\end{prop}

\begin{proof}
	By Lemma \ref{lemma:direct_sum_invariant}, we can suppose that $\rho$ has the form
	\begin{equation*}
		\rho = \bigoplus_{i=1}^r \rho_i^{\oplus m_i} \oplus \bigoplus_{i=r+1}^{s} (\rho_i \oplus \Phi_1(\rho_i))^{\oplus m_i}
	\end{equation*}
	where $\rho_i$ is conjugated to $\Phi_1(\rho_i)$ for $i \in \{1,\dots, r\}$.
	
	By \Cref{thm:char_var_real_slnc}, if $i \in \{1,\dots,r \}$, $\rho_i$ is conjugated to a representation $\rho'_i$ with values in some $\mathrm{GL}(k,\RR)$ or $\mathrm{GL}(k,\HH)$.
	Furthermore, if $i \in \{r+1, \dots , s \}$, by Lemma \ref{lemma:rho_plus_rho_conj}, each representation $\rho_i \oplus \Phi_1(\rho_i)$ is conjugated to a representation $\sigma_i$ with values in some $\mathrm{GL}(k,\HH)$.
	Therefore, $\rho$ is conjugated in $\mathrm{GL}(n,\CC)$ to the representation
	\begin{equation*}
		 \bigoplus_{i=1}^r (\rho'_i)^{\oplus m_i} \oplus \bigoplus_{i=r+1}^{s} \sigma_i^{\oplus m_i}
	\end{equation*}
	that takes values in a product of copies of $\mathrm{GL}(r_i , \RR)$ or $\mathrm{GL}(r_i , \HH)$. Thus, by rearranging the terms, we obtain $r_1,r_2 \in \NN$ such that $\rho$ is conjugated to $\rho' \oplus \rho''$ where $\rho' \in \mathrm{Hom}(\Gamma, \mathrm{GL}(r_1,\RR))$ and $\rho'' \in \mathrm{Hom}(\Gamma, \mathrm{GL}(r_2 , \HH))$.
\end{proof}

Finally, we prove the same type of statement for the involution $\Phi_2$, which corresponds to unitary groups.

\begin{prop}\label{prop:reducible_fixed_phi2}
	Let $\rho \in \mathrm{Hom}(\Gamma,\mathrm{GL}(n,\CC))$ be a semi-simple representation. If $\Phi_2(\chi_\rho) = \chi_\rho$, then there exist $p,q \in \NN$ with $p+q =n$ such that $\rho$ is conjugated to a representation with values in $\mathrm{U}(p,q)$.
\end{prop}

\begin{proof}
	By Lemma \ref{lemma:direct_sum_invariant}, we can suppose that $\rho$ has the form
	\begin{equation*}
		\rho = \bigoplus_{i=1}^r \rho_i^{\oplus m_i} \oplus \bigoplus_{i=r+1}^{s} (\rho_i \oplus \Phi_1(\rho_i))^{\oplus m_i}
	\end{equation*}
	where $\rho_i$ is conjugated to $\Phi_2(\rho_i)$ for $i \in \{1,\dots, r\}$.
	
	By \Cref{thm:char_var_real_slnc}, if $i \in \{1,\dots,r \}$, $\rho_i$ is conjugated to a representation $\rho'_i$ with values in some $\mathrm{U}(p_i,q_i)$.
	Furthermore, if $i \in \{r+1, \dots , s \}$, by Lemma \ref{lemma:rho_plus_rho_phi2}, each representation $\rho_i \oplus \Phi_2(\rho_i)$ is unitary for a Hermitian form of signature $(r_i,r_i)$, and thus is conjugated to a representation $\sigma_i$ with values in $\mathrm{U}(r_i,r_i)$.
	Therefore, $\rho$ is conjugated in $\mathrm{GL}(n,\CC)$ to the representation
	\begin{equation*}
		\bigoplus_{i=1}^r (\rho'_i)^{\oplus m_i} \oplus \bigoplus_{i=r+1}^{s} \sigma_i^{\oplus m_i}
	\end{equation*}
	that takes values in a product of copies of $\mathrm{U}(p_i,q_i)$. Thus, by grouping the signatures, there exist $p,q \in \NN$ with $p+q =n$ such that $\rho$ is conjugated to a representation with values in $\mathrm{U}(p,q)$.
\end{proof}

 \bibliographystyle{alpha}
 \bibliography{character_varieties.bib}

\end{document}

%% file: preamble_eng.tex
\usepackage[utf8]{inputenc}
\usepackage[T1]{fontenc}

\usepackage{babel}

\usepackage[final]{graphicx} 
\usepackage{numprint} 

\usepackage{geometry} 
\usepackage{xcolor} 

\usepackage{rotating}

\usepackage[nottoc, notlof, notlot]{tocbibind} 

\usepackage{verbatim} 

\usepackage{hyperref} 

\usepackage{float} 
\usepackage{placeins} 

\usepackage{caption}
\usepackage{subcaption}

\usepackage{multicol} 

\usepackage{wrapfig} 

\usepackage{mathtools}
\usepackage{amsfonts,amssymb,amsthm}

\usepackage{array}
\usepackage{multirow} 

\usepackage{cleveref} 

\theoremstyle{plain}
\newtheorem{thm}{Theorem}[section]
\newtheorem{lemme}[thm]{Lemma}
\newtheorem{prop}[thm]{Proposition}
\newtheorem{cor}[thm]{Corollary}
\newtheorem*{thm*}{Theorem}
\newtheorem*{prop*}{Proposition}

\theoremstyle{definition}
\newtheorem{defn}[thm]{Definition}
\newtheorem*{defn*}{Definition}

\newtheorem*{conj*}{Conjecture}

\newtheorem{rem}[thm]{Remark}

\newcommand*\conjug[1]{\overline{#1}}

\renewcommand{\Im}{\mathrm{Im}}
\renewcommand*\setminus{\ensuremath{{}-{}}}

\newcommand{\hookuparrow}{\mathrel{\rotatebox[origin=c]{90}{$\hookrightarrow$}}}

\newcommand{\HH}{\mathbb{H}}
\newcommand{\CC}{\mathbb{C}}
\newcommand{\RR}{\mathbb{R}}

\newcommand{\NN}{\mathbb{N}}

\newcommand*{\transp}[2][-3mu]{\ensuremath{\mskip1mu\prescript{\smash{\mathrm t\mkern#1}}{}{\mathstrut#2}}}%